\documentclass[11pt]{article}

\usepackage{amsmath,amsfonts,amssymb,dsfont,mathrsfs}
\usepackage{graphicx, amsthm,color}

\usepackage{tikz}
\usepackage{tikz-3dplot}
\usetikzlibrary{calc,patterns,angles,quotes}
\usetikzlibrary{babel}
\usepackage{pgfplots}
\usetikzlibrary{backgrounds, intersections}
\usepgfplotslibrary{fillbetween}
\usetikzlibrary{arrows,automata}

\providecommand{\U}[1]{\protect\rule{.1in}{.1in}}

\newtheorem{theorem}{Theorem}[section]
\newtheorem{corollary}[theorem]{Corollary}
\newtheorem{lemma}[theorem]{Lemma}
\newtheorem{proposition}[theorem]{Proposition}
\theoremstyle{definition}
\newtheorem{definition}[theorem]{Definition}
\newtheorem{remark}[theorem]{Remark}
\newtheorem{example}[theorem]{Example}

\let\epsilon\varepsilon

\newcommand{\Rex}{\overline{\mathbb{R}}}

\providecommand{\dom}{\mathop{\rm dom}\nolimits}
\providecommand{\argmin}{\mathop{\rm argmin}\nolimits}

\DeclareFontFamily{U}{txsyc}{}
\DeclareFontShape{U}{txsyc}{m}{n}{
   <-> txsyc
}{}
\DeclareFontShape{U}{txsyc}{bx}{n}{
   <-> txbsyc
}{}
\DeclareFontShape{U}{txsyc}{l}{n}{<->ssub * txsyc/m/n}{}
\DeclareFontShape{U}{txsyc}{b}{n}{<->ssub * txsyc/bx/n}{}
\DeclareSymbolFont{symbolsC}{U}{txsyc}{m}{n}
\SetSymbolFont{symbolsC}{bold}{U}{txsyc}{bx}{n}
\DeclareFontSubstitution{U}{txsyc}{m}{n}
\DeclareMathSymbol{\df}{\mathrel}{symbolsC}{"42}
\DeclareMathSymbol{\fd}{\mathrel}{symbolsC}{"43}
\DeclareMathSymbol{\lJoin}{\mathrel}{symbolsC}{"58}
\DeclareMathSymbol{\rJoin}{\mathrel}{symbolsC}{"59}

\newcommand{\f}[2]{\frac{#1}{#2}}

\newcommand{\cK}{\mathcal{K}}

\newcommand{\cP}{\mathcal{P}}

\newcommand{\cV}{\mathcal{V}}

\newcommand{\RR}{\mathbb{R}}

\newcommand{\lt}{\left}

\newcommand{\rt}{\right}

\newcommand{\fo}{\forall\ }

\newcommand{\st}{\,:\,}

\newcommand{\ind}[1]{\mathds{1}_{\! #1}}

\newcommand{\bq}{\begin{eqnarray*}}
\newcommand{\bqn}[1]{\begin{eqnarray}\label{#1}}
\newcommand{\eq}{\end{eqnarray*}}
\newcommand{\eqn}{\end{eqnarray}}

\newcommand{\ttsim}{\raise.17ex\hbox{$\scriptstyle\mathtt{\sim}$}}

\newcommand{\kh}{\kern .08em}

\newcommand{\comment}[1]{}

\newcommand{\globalSlope}{\mathscr{G}}

\definecolor{ForestGreenWeb}{rgb}{0.13, 0.55, 0.13}
\begin{document}

\begin{center}
{\LARGE Descent modulus and applications}
\end{center}

\smallskip

\begin{center}
{\large \textsc{Aris Daniilidis, Laurent Miclo, David Salas}}
\end{center}

\medskip

\textbf{Abstract.} The norm of the gradient $\|\nabla f(x)\|$ measures the maximum descent of a real-valued smooth function $f$ at $x$. For (nonsmooth) convex functions, this is expressed by the distance $\mathrm{dist}(0,\partial f(x))$ of the subdifferential to the origin, while for general real-valued functions defined on metric spaces by the notion of metric slope $|\nabla f|(x)$. In this work we propose an axiomatic definition of descent modulus $T[f](x)$ of a real-valued function $f$  at every point $x$, defined on a general (not necessarily metric) space. The definition encompasses all above instances as well as average descents for functions defined on probability spaces. We show that a large class of functions are completely determined by their descent modulus and corresponding critical values.
This result is already surprising in the smooth case: a one-dimensional information (norm of the gradient) turns out to be almost as powerful as the knowledge of the full gradient mapping. In the nonsmooth case, the key element for this determination result is the break of symmetry induced by a downhill orientation, in the spirit of the definition of the metric slope. The particular case of functions defined on finite spaces is studied in the last section. In this case, we obtain an explicit classification of descent operators that are, in some sense, typical.



\vspace{0.40cm}


\vspace{0.40cm}

\textit{Secondary} 30L15, 37C10, 58G32. 

\tableofcontents

\newpage

\section{Introduction\label{sec01:Intro}}

In \cite{BCD2018} the following surprising result was obtained: two 
$\mathcal{C}^{2}$-smooth convex bounded from below functions defined on a
Hilbert space $\mathcal{H}$ are equal up to an additive constant, provided
they have the same modulus of derivative at every point. In other words, for
this class of functions, equality of moduli of derivatives 
($\Vert \nabla f\Vert =\Vert \nabla g\Vert $) implies equality of the derivatives ($\nabla f=\nabla g$). An alternative way to announce this result is to say that the
operator
\begin{equation}
f\mapsto \Vert \nabla f\Vert   \label{eq:d1}
\end{equation}
determines the function $f$ (modulo a constant) for the class of $\mathcal{C}^{2}$-smooth convex and bounded from below functions defined on the Hilbert space $\mathcal{H}$.\smallskip 

The above result has been extended in \cite{PSV2021} to the class of convex
continuous bounded from below functions on a Hilbert space $\mathcal{H}$. A
further extension for functions defined on an arbitrary Banach space $X$ has
been achieved in \cite{TZ2022}. In both cases, the operator
\begin{equation}
f\mapsto d(0,\partial f(x))\quad \text{(remoteness of the subdifferential)}
\label{eq:d2}
\end{equation}
determines the function $f$ (modulo a constant) for the class of convex
continuous and bounded from below functions on a Banach space $X$.\smallskip 

In \cite{DS2022} the authors worked on an arbitrary metric space $(X,d)$.
 Using the notion of metric slope $|\nabla f|$, they established the following
result: two continuous coercive functions $f,g:X\rightarrow~\mathbb{R}$
are equal, provided they have the same metric slope ($|\nabla f|=|\nabla g|$) and coincide on the (common) critical set $S=|\nabla f|^{-1}(0)=|\nabla
g|^{-1}(0).$ (We refer to Section~\ref{sec02:Pre} for notation and relevant
definitions; see also Subsection~\ref{ssec:2.1} for a more detailed description of the above results.)
Denoting by $\mathcal{K}(X)$ the class of continuous coercive
functions on $X$ (the exact definition of coercivity will be given in~\eqref{sec02:eq:Coercive}), we consider the following equivalent relation: $f\sim g$ if and
only if $f$ and $g$ have the same (metric) critical set and their values coincide there up to a constant, that is,
\begin{equation*}
f\sim g \quad \iff \quad S=|\nabla f|^{-1}(0)=|\nabla
g|^{-1}(0)\text{ and } f\big|_S - g\big|_S = c, \,\,\text{for some }c\in\mathbb{R}.
\end{equation*}
Then, the aforementioned result of \cite{DS2022} asserts that the operator 
\begin{equation}
f\mapsto |\nabla f|  \label{eq:d3}
\end{equation}
is injective on $\mathcal{K}(X)/\!\!\sim $, that is, it determines
functions $f\in \mathcal{K}(X)$ modulo the equivalent relation~$\sim$. \smallskip 

We refer to all above results as determination results on a specific class
of function (modulo a natural equivalent relation). Although the last result
is formulated in an abstract metric space and is quite general, we will show in this work 
that a deeper result is hidden. Namely, the metric structure
is ostensibly required to define the determining operator, but is not really
paramount: the quantities $\Vert \nabla f(x)\Vert $ (in the smooth case) and 
$|\nabla f|(x)$ (in a metric setting) express the steepest descent of $f$ at
a given point $x,$ however, this is not the only possible choice to deal with
descent properties. For instance, one can also consider average descent (based
on some probability measure on the space $X$) and emancipate dependence from
the metric structure. The above leads to a definition of an abstract
descent operator (which does not rely on a distance or even a
topology). This abstract scheme, developed in Section~\ref{sec03:Axiomatic}, encompasses several instances of descent-type operators, 
in particular both paradigms of steepest descent and average descent. 
In Section~\ref{sec04:Averaged} we study general diffusion processes in metric spaces and show that
asymmetrization (via downhill orientation) is the key property to obtain a
determination result, in a complete analogy to the asymmetric definitions of~\eqref{eq:d2}--\eqref{eq:d3}. In the last section we consider the particular
case of descent operators in finite dimensional spaces and obtain an 
explicit classification of a broad subfamily of these operators.

\section{Notation and Preliminaries}\label{sec02:Pre}

We set $\Rex = \mathbb{R} \sqcup \{-\infty,+\infty\}$ and $\Rex_+ = \mathbb{R}_+\sqcup\{+\infty\}$. For any $a\in\mathbb{R}$ we set  $a_+=\max\{a,0\}$. For two real numbers $r,s\in \mathbb{R}$, we write $r\wedge s := \min\{ r,s \}$ and $r\vee s := \max\{r,s \}$. Given a nonempty set $X$ and a function $f:X\to \Rex$ we define the domain of $f$ as follows:
\[
\mathrm{dom }(f) :=\{ x\in X:\, f(x)<+\infty \}\,.
\]
Given $\alpha\in \mathbb{R}$, we write
\begin{align*}
    [f\leq \alpha] &:= \{ x\in X\ :\ f(x) \leq \alpha \}\\
    [f < \alpha] &:= \{ x\in X\ :\ f(x) < \alpha \}
\end{align*}
to denote the sublevel set and strict sublevel set of $f$ at value $\alpha$. The sets $[f = \alpha]$, $[f\geq \alpha]$ and $[f>\alpha]$ are defined analogously. \smallskip \newline
We shall often equip the set $X$ with a topology, denoted by $\tau$. In this case, we denote by $\mathcal{B}(X)$ the $\sigma$-algebra of the Borel subsets of the topological space $(X,\tau)$. \smallskip \newline
We say that a function $f:X\to\mathbb{R}\cup\{+\infty\}$ is \textit{$\tau$-lower semicontinuous} if all sublevel sets $[ f\leq a]$, $a\in\mathbb{R}$, are $\tau$-closed subsets of $X$. The function $f$ is called \textit{$\tau$-coercive} if
\begin{equation}\label{sec02:eq:Coercive}
\text{ for every } \alpha\in (-\infty, \sup f),\quad \text{the sublevel set } \,\, [f\leq \alpha]\text{ is $\tau$-compact}.
\end{equation} 
We simply call $f$ lower semicontinuous (respectively, coercive), when no ambiguity on the topology occurs. Notice that the above definition of coercivity encompasses in particular all constant functions. 
\smallskip

We further denote by
\begin{equation} 
\mathcal{C}(X) \text{ the space of continuous functions from } X \text{ to } \mathbb{R}
\end{equation}
and  we define the subclass of coercive continuous functions by
\begin{equation}
    \mathcal{K}(X) := \{ f\in \mathcal{C}(X)\ :\ f\,\text{ is }\tau\text{-coercive}\}.
\end{equation}
If $(X,\tau)$ is compact, then every lower semicontinuous functional is coercive and $\mathcal{K}(X) = \mathcal{C}(X)$.\smallskip\newline 
Let $\mathcal{L}_n$ stand for the usual Lebesgue measure over $\mathbb{R}^n$ and let $B_n(x,r)$ (respectively, $\overline{B}_n(x,r)$) be the open (respectively, closed) ball centered at $x\in\mathbb{R}^n$ of radius $r>0$. We also denote by $\mathbb{B}_n$ (respectively, $\mathbb{S}_n$) the closed unit ball (respectively, unit sphere) of $\mathbb{R}^n$. If there is no ambiguity, we omit the subindex $n$ for each of the elements above.  It is well known (see, e.g., \cite{SV1989}) that the $n$-dimensional volume of the ball of radius $r>0$ in $\mathbb{R}^n$ is given by
\begin{equation}\label{eq:VolumeNBall}
    \mathcal{L}_n(B(0,r)) = \frac{\pi^{n/2}}{\Gamma\left( \tfrac{n}{2} + 1 \right)}r^n \,,
\end{equation}
where $\Gamma$ stands for the gamma function. In particular, the volume of the $n$-dimensional ball $B(0,r)$ is proportional to $r^n$. For any (affine) subspace $W$ of $\mathbb{R}^n$, we denote by $\dim(W)$ its (affine) dimension. \smallskip\newline
We say that a family $\mathcal{F}$ of real-valued functions is a \textit{cone} if for every $f\in\mathcal{F}$ and $r\geq 0$ we have $rf \in \mathcal{F}$. In addition, we say that $\mathcal{F}$ is a \emph{translation cone} if it is closed by translations (that is, for every $f\in\mathcal{F}$ and every constant $c\in\mathbb{R}$, we have that $f+c\in\mathcal{F}$).
 Clearly, the set  $\mathcal{K}(X)$  of coercive continuous functions is a 
translation cone in $\mathcal{C}(X)$. \smallskip \newline 
For an operator $T:\mathcal{F}\to (\Rex_+)^X$, we define  its \textit{domain} 
\[
\dom(T):=\{ f\in \mathcal{F}: \quad T[f](x)<+\infty,\,\,\text{for all } x\in X\}. 
\]
If $(X,d)$ is a metric space, we define the metric slope $|\nabla f|(x)$
of an extended real-valued function $f:X\rightarrow \mathbb{R}\cup \{+\infty
\}$ at the point $x\in \mathrm{dom}(f)$ as follows: 
\begin{equation}\label{eq:metslo}
|\nabla f|(x):=\left\{ 
\begin{array}{cl}
\underset{y\rightarrow x}{\limsup }\frac{\left\{ f(x)-f(y)\right\} _{+}}{d(y,x)}, & \text{ if }x\text{ is not isolated,} \\ 
0, & \text{ otherwise.}
\end{array}
\right.
\end{equation}
In the same setting, the global slope $\globalSlope[f](x)$ is defined as follows:
\begin{equation*}
\globalSlope[f](x):=\,\underset{y\in X}{\sup }\frac{\left\{ f(x)-f(y)\right\} _{+}}{d(y,x)}.
\end{equation*}

Notice that $\globalSlope[f](x)=0$ if and only if $x\in \argmin f$ (\textit{i.e.} $x$ is global minimum of $f$),
while $|\nabla f|(x)=0$ whenever $x$ is a local minimum of $f$. The notions of metric slope (also known as strong slope) and global slope are well known in the literature (see, e.g., \cite{AGS2008} and the references therein). \smallskip\newline

Let us now assume that $X$ is a Banach space with dual $X^{\ast }.$ It is
well-known that if $f:X\rightarrow \mathbb{R}$ is a smooth function, then 
\begin{equation*}
|\nabla f|(x)=\Vert \nabla f(x)\Vert .
\end{equation*}
In the nonsmooth setting, if $f:X\rightarrow \mathbb{R}\cup \{+\infty \}$ a
lower semicontinuous \textit{convex} function, its (Fenchel-Moreau)
subdifferential $\partial f(x)$ at $x\in \mathrm{dom}(f)$ is defined as
follows:
\begin{equation*}
\partial f(x)=\{p\in X^{\ast }:\forall y\in X,\,f(y)\geq f(x)+\langle
p,y-x\rangle \}.
\end{equation*}
It is well-known that $\partial f(x)$ is a closed convex set and it is nonempty if $x$ is a point of continuity of $f$ (see, e.g.,  \cite{RockafellarWets1998}). Moreover, it is known (see, e.g., \cite[Proposition 1.4.4]{AGS2008}) that for any lower semicontinuous convex function over a Banach space, one has
\begin{equation}
|\nabla f|(x)=\globalSlope[f](x)=d(0,\partial f(x)).  \label{eq:global-local}
\end{equation}

\subsection{State-of-the-art}\label{ssec:2.1}

The derivative of a smooth function recovers, up to an additive constant,
the function through integration. In the nonsmooth case, Rockafellar \cite{Rock70} showed that every lower semicontinuous convex function can be
represented through its subdifferential by means of a \textit{nonsmooth
integration}. This result has been refined in \cite{BD2015} for Banach
spaces with the Radon-Nikodym property, provided the function satisfies a
mild coercivity property (namely, the asymptotic cone of its epigraph is
epi-pointed). In this latter case, a partial knowledge of the
subdifferential mapping $\partial f$ is sufficient to recover the function
up to a constant.\smallskip

Historically, this integration result was first stated as a determination result: For every two proper convex lower semicontinuous functions $f,g:X\to\mathbb{R}$ over a Banach space $X$, one has that
\begin{equation}
	(\partial f(x) = \partial g(x), \forall x\in X) \implies f = g + c,\,\, \text{for some }c\in\mathbb{R}.
\end{equation}
This result was first obtained in Hilbert spaces by Moreau \cite{Moreau1965}, and generalized to Banach spaces one year later by Rockafellar \cite{Rockafellar1966}.  A more general result was established by Brezis for monotone operators \cite{BrezisOperateurs1973}, where the same determination result can be obtained in Hilbert spaces only in terms of the element of minimal norm of the subdifferential, that is, 
\begin{equation}
	( \mathrm{proj}(0,\partial f(x)) = \mathrm{proj}(0,\partial g(x)), \forall x\in \mathcal{H}) \implies f = g + c,\,\, \text{for some }c\in\mathbb{R},
\end{equation}
where $\mathrm{proj}(x,A)$ denotes the metric projection of $x$ onto the set $A$. Notice that knowledge of a full gradient $\nabla f(x)$
(respectively, subdifferential $\partial f(x)\subset X^{\ast }$, or $ \mathrm{proj}(0,\partial f(x))\in\mathcal{H}$) at many (all) $x\in X$ 
is already a rich information: at every such point $x$ we need to know a vector 
(respectively, a set of vectors). Nonewithstanding, it has recently
become clear that much less information (namely, a scalar) is often sufficient if our objective is only to determine functions 
(rather than recovering them via an explicit formula). This is resumed below:

\subsubsection{Determination of convex functions\label{ssec:3.1}}

Let $\mathcal{H}$ be a Hilbert space and $f:\mathcal{H}\rightarrow 
\mathbb{R}$ be a $\mathcal{C}^{2}$-smooth convex and bounded from below
function. Set $V(x)=\frac{1}{2}\Vert \nabla f(x)\Vert ^{2}$ and consider the
second-order dynamical system 
\begin{equation}
\ddot{x}(t)=\nabla V(x(t)),  \label{eq:DS2}
\end{equation}%
with initial condition $x(0)=x_{0}\in \mathcal{H}$. It has been shown in 
\cite{BCD2018} that every evanescent solution of~\eqref{eq:DS2} (that is,
every solution satisfying $\Vert \dot{x}(t)\Vert \rightarrow 0$ and $\Vert
\nabla V(x(t))\Vert \rightarrow 0,$ as $t\rightarrow \infty $) is solution
of the first order gradient system:
\begin{equation}
\left\{ 
\begin{array}{l}
\dot{x}(t)=-\nabla f(x(t))\smallskip \\ 
x(0)=x_{0}.
\end{array}
\right.  \label{eq:DS1}
\end{equation}
On the other hand, $\mathcal{C}^{2}$-smoothness assumptions guarantees that~\eqref{eq:DS1} has unique solution. The fact that $f$ is bounded from below
ensures that this solution is evanescent. By straightforward differentiation
we deduce that this solution is also solution of the second-order system~\eqref{eq:DS2}, therefore \eqref{eq:DS2} and \eqref{eq:DS1} have the same
solutions. Since \eqref{eq:DS2} depends only on $\Vert \nabla f\Vert $
(rather than on $\nabla f$), the following conclusion has been obtained:

\begin{itemize}
\item If $f,g:\mathcal{H}\rightarrow \mathbb{R}$ are two $\mathcal{C}^{2} $-smooth convex bounded from below functions, then 
\begin{equation}
\Vert \nabla f\Vert =\Vert \nabla g\Vert \,\Longleftrightarrow \,\nabla
f=\nabla g\,\Longleftrightarrow \,f=g+c,\,\text{for some }c\in \mathbb{R}. \label{eq:tahar}
\end{equation}
\end{itemize}

Notice that $\mathcal{C}^{2}$-smoothness was required in order to define
property \eqref{eq:DS2}. However, this assumption can be relaxed to $\mathcal{C}^{1}$-smoothness, assuming existence of minimizers \cite{Baillon2018}. This is based on the remark that
\begin{equation}
\Vert \nabla f\Vert =\Vert \nabla g\Vert \,\,\Longleftrightarrow \,\langle
\nabla (f+g),\nabla (f-g)\rangle   \label{eq:baillon}
\end{equation}
which ensures in turn that the function $f-g$ is constant among the gradient
orbits of the (convex) function $f+g.$ Since each such orbit \textit{lands}
on the (common) set $S$ of minimizers of the functions $f,$ $g$ and $f+g,$
and since $f-g$ is constant there (with value $\min f-\min g$), the result
follows.\smallskip \newline
A generalization of \eqref{eq:tahar} has been carried out in \cite{PSV2021}, where smoothness assumption has been replaced by continuity.

\begin{itemize}
\item If $f,g:\mathcal{H}\rightarrow \mathbb{R}$ are two convex
continuous and bounded from below functions, then 
\begin{equation}
\begin{aligned}	
d(0,\partial f(x))=d(0,\partial g(x)),\;\text{for all }x\in \mathcal{H}\,\,&\Longleftrightarrow \,\,\partial f=\partial g\\
&\Longleftrightarrow \,\,f=g+c,\,\,\text{for some }c\in\mathbb{R}. 
\end{aligned} \label{eq:salas}
\end{equation}
\end{itemize}

To achieve the above result, the authors studied the subgradient system $\dot{x}(t)\in -\partial f(x(t))$ and showed that in this case the assumption 
$d(0,\partial f(\cdot ))\geq d(0,\partial g(\cdot ))$ yields $f\geq g+c$. This is proven based on two key observations: First, the solution $x(t)$  is not only  a minimizing curve for $f$ (i.e., $f(x(t))\to \inf f$ as $t\to +\infty$), but it is also a minimizing curve for $g$.  Second, the chain rule of the convex subdifferential entails that $(f-g)$ is nonincreasing along $x(t)$. Thus, one can consider $c = \inf f - \inf g$. After proving this comparison principle,  \eqref{eq:salas} follows by symmetry.

\subsubsection{Determination in metric spaces} \label{ssec:3.2}

Convexity assumption was important for the proofs of \eqref{eq:tahar} and 
\eqref{eq:salas}. However, the outlined proof via \eqref{eq:baillon} was
using convexity for two factors: to conclude that every steepest descent curve lands on a critical point (i.e. has an accumulation point in the set of critical points), and that every critical point is global
minimizer, that is,
\begin{equation}\label{eq:stalin}
\mathrm{Crit\ }f=\left\{ x\in X:\ \nabla f(x)=0\right\} =\arg \min f.
\end{equation}
Assuming~\eqref{eq:stalin} and some coercivity condition (instead of convexity), the same argument leads to the following result:
\begin{itemize}
\item If $f,g:\mathcal{H}\rightarrow \mathbb{R}$ are two $\mathcal{C}^{1}$-smooth coercive functions, then:
\begin{equation*}
\left. 
\begin{array}{c}
\Vert \nabla f\Vert =\Vert \nabla g\Vert \text{ \quad and}\medskip \\ 
\mathrm{Crit\ }f=\arg \min f=\arg \min g\neq \emptyset 
\end{array}
\right\} \,\Longrightarrow \;\,f=g+c,\,\,\text{for some }c\in\mathbb{R}.
\end{equation*}

\end{itemize}

On the other hand, all results mentioned in the previous subsection are
strongly based on (sub)gradient dynamical systems and the Hilbertian
structure of the space. Quite surprisingly, it turns out that this structure
is eventually not required. Indeed, in the recent work \cite{TZ2022}, the
result \eqref{eq:salas} has been extended to arbitrary Banach spaces,
through a completely different approach, which was based on the notion of 
\textit{countable orderable sets} introduced in \cite{GZ1992}. In that work
the authors establish that two continuous and bounded from below functions
$f,g:X\rightarrow \mathbb{R}$ defined on a metric space $(X,d)$ and with
finite global slopes are equal up to a constant, provided they have the same
global slopes at every point. In other words:
\begin{equation}
\globalSlope[f]=\globalSlope[g]<+\infty \,\Longrightarrow \,f=g+c,\text{ for some }c\in\mathbb{R}.  \label{eq:Glodet}
\end{equation}
The key technique to achieve such a result is the construction of a minimizing sequence by means of the global slope. The construction is based on the following result (proved in \cite{TZ2022}):  for every sequence $\{x_i\}_i$  of the metric space $(X,d)$ and for every proper extended real-valued function $f:X\to \mathbb{R}\sqcup\{\infty\}$ it holds:
\begin{equation}
	\left( \lim_{i\to \infty}\globalSlope[f](x_i) = 0 \text{ and }\sum_{i=1}^\infty \globalSlope[f](x_i)d(x_i,x_{i+1}) < \infty\right) \implies \liminf_{i\to \infty} f(x_i ) = \inf_X f.
\end{equation}

 Although the setting is quite general (metric spaces), the notion of global
slope is rather restrictive, since it does not coincide with the modulus of
the derivative in the smooth case. But this notion is a very good fit for
convex functions defined on a general Banach space $X.$ In this case, 
combining \eqref{eq:global-local} with \eqref{eq:Glodet} yields a generalization of~\eqref{eq:salas}.\smallskip 

In an independent work~\cite{DS2022}, the authors considered the local notion of metric slope and established the following comparison result for the class of continuous coercive functions. In what follows we denote by 
\begin{equation*}
\mathrm{Crit }f=\left\{ x\in X:\ |\nabla f|(x)=0\right\} 
\end{equation*}
the set of (metrically) critical points.

\begin{proposition}[slope comparison] \label{prop:metric-comp}
Let $(X,d)$ be a metric space and  $f,g:X\rightarrow \mathbb{R}$ be two continuous coercive functions.
Assume that

\begin{itemize}
\item[(i).] $|\nabla f|(x)>|\nabla g|(x),\;\,$for all $x\in X\diagdown 
\mathrm{Crit }f.$ ;\quad and

\item[(ii).] $f(x)>g(x)$, for all $x\in \mathrm{Crit }f.$
\end{itemize}

Then, $f>g$.
\end{proposition}

The proof was obtained by reasoning to contradiction and using discrete
iterations and transfinite induction. The following determination result was
then obtained as consequence of Proposition~\ref{prop:metric-comp}.

\begin{theorem}[Determination in metric spaces]
\label{thm:metric-det}Let $(X,d)$ be a metric space and $f,g:X\rightarrow \mathbb{R}$ be two continuous coercive functions. Assume that

\begin{itemize}
\item[(i).] $|\nabla f|(x)=|\nabla g|(x)<+\infty ,\;\,$for all $x\in X$
;\quad and

\item[(ii).] $f(x)=g(x)$, for all $x\in \mathrm{Crit\ }f.$
\end{itemize}

Then, $f=g$.
\end{theorem}

The above result asserts that information on the metric slope $|\nabla f|$
and critical values is sufficient to determine every continuous coercive
functions with finite slope (therefore, in particular, every Lipschitz
continuous coercive functions). Taking into account the pathologies that
prevail Lipschitz functions, the above statement appears close to be
optimal: In \cite{DS2022} several counterexamples are presented to show the
pertinence of the assumptions. This being said, there is still room for 
improvements: indeed, assuming $X$ is a complete metric space, it seems 
plausible to relaxe coercivity/compactness assumption (which is required in the current proof), 
by an alternative assumption ensuring the existence of appropriate descent (generalized) sequences that link any point to the critical set.

\subsection{Description of the current work}

Revising the arguments employed in \cite{DS2022} for the proof of
Proposition~\ref{prop:metric-comp} and Theorem~\ref{thm:metric-det} we
observe that continuity and coercivity are topological notions, while the
metric structure of $(M,d)$ is only required in order to define the metric slope,
see~\eqref{eq:metslo}. In particular, one can assume continuity and
coercivity with respect to another topology (not related to the given distance $d$) 
and the topological part of the proof can be completely
decoupled.\smallskip \newline
In this work we show that a similar result to Theorem~\ref{thm:metric-det}
holds for any topological space equipped with a Borel probability measure $\mu $, if we replace the metric slope $|\nabla f|(x)$ (corresponding to the
steepest descent at $x$) by the $\mu $-\textit{average descent }$T_{\mu
}(f)(x)$ \textit{at }$x$ given by
\begin{equation*}
T_{\mu }(f)(x):=\int_{X}\left\{ f(x)-f(y)\right\} _{+}d\mu
(y)=\int_{[f\leq f(x)]}\left[ f(x)-f(y)\right] d\mu (y).
\end{equation*}
More generally, we introduce an abstract descent operator $T[f]$ (see
Definition~\ref{sec03:def:ModulusOfDescent}) that encompasses both metric
and global slopes (in metric spaces) and average descent (in probability
spaces) as well as many other instances, see Subsection~\ref{sec03-03:Stability} for further examples and stability properties of this
operator. We then establish an abstract determination result (Theorem~\ref{sec03:thm:Determination}) revealing that the metric structure is neither
minimal nor optimal framework, as hinted by the topological and metric decoupling observed in \cite{DS2022}.
In Section~\ref{sec04:Averaged} we study general stochastic processes in metric
spaces and define adequate oriented operators (particular instances of
Definition~\ref{sec03:def:ModulusOfDescent}) that allow to obtain
determination results. Finally, in Section~\ref{sec:5} we introduce an
equivalence relation among descent moduli for functions $f\in \mathbb{R}^{\mathcal{V}}$ defined in finite spaces $\mathcal{V}$ and show that a typical
descent modulus is equivalent to a steepest descent with respect to a
prescribed active neighborhood system (see Theorem~\ref{critmap}).



\section{Descent modulus: definition, properties and main examples} \label{sec03:Axiomatic}


Let $\mathcal{F}$ be a family of functions
from a nonempty set $X$ to $\mathbb{R}$. For an operator $T:\mathcal{F}\rightarrow (\overline{\mathbb{R}}_{+})^{X}$, we define its 
\textit{domain} 
\begin{equation}\label{eq:domT}
\dom(T):=\{f\in \mathcal{F}:\quad T[f](x)<+\infty ,\,\,\text{for all }x\in
X\}.
\end{equation}
We also define the set $\mathcal{Z}_{T}(f)$ of \textit{$T$-critical points}
of $f\in \mathcal{F}$ as follows: 
\begin{equation}
\mathcal{Z}_{T}(f)=\{x\in X\ :\ T[f](x)=0\}.  \label{eq:ZTf}
\end{equation}
(Note that every $T$-critical point of $f$ is a global minimizer for the
function $T[f]$.)\smallskip\newline
In this section we give an axiomatic definition of an abstract descent operator,
that is, an operator $T$ acting on (a certain class of) functions $f$ from $X$ to $\mathbb{R}$. This operator associates to each point $x\in X$ an extended nonnegative number $T[f](x)\in \mathbb{R}\cup \{+\infty
\}$ which corresponds to an abstract measure of descent (henceforth called 
\textit{descent modulus}) of $f$ at $x$. \smallskip \newline
The required properties of this abstract definition will be kept minimal to
encompass several instances stemming from classical and variational
analysis, metric geometry and stochastic processes: in particular, the
metric slope (used in~\cite{DS2022}), the global slope (used in~\cite{TZ2022}) and the notion of average descent (that will be discussed later in this
work) are all captured by the proposed abstract scheme. 

\subsection{Axiomatic definition\label{sec03-01:Construction}}

Let $\mathcal{F}$ be a translation cone in the space of
functions from a nonempty set $X$ to $\mathbb{R}$. 

\begin{definition}
[Abstract descent modulus]\label{sec03:def:ModulusOfDescent} Let
$T:\mathcal{F}\rightarrow(\overline{\mathbb{R}}_{+})^{X}$ be a nonlinear
operator.\\ We say that:\smallskip\newline\textrm{(D1).} $T$ \textit{preserves global
minima}, if for every function $f\in\mathcal{F}$ and $x\in X$ we have
\[
x\in\argmin f\implies x\in\mathcal{Z}_{T}(f)\,.
\]
\textrm{(D2).} $T$ \textit{is monotone at $x$},
if for every $f,g\in\mathcal{F}$ we have:
\begin{equation}
\forall z\in X:\,(f(x)-f(z))_{+}\geq(g(x)-g(z))_{+}\,\,\implies\,\,T[f](x)\geq
T[g](x).\label{eq:ad}
\end{equation}
\textrm{(D3).} $T$ \textit{is scalar-monotone at $x$}, if for every function
$f\in\mathcal{F}$ and $r>1$, we have
\[
0<T[f](x)<+\infty\,\implies\,T[f](x)<T[rf](x).
\]
The operator $T$ is called (scalar) monotone, if it is (scalar) monotone at
every $x\in X$. \smallskip\newline
We say that $T$ is a \textit{descent modulus} for the class
$\mathcal{F}$ if properties (\textrm{D1})--(\textrm{D3}) hold, that is, if $T$ is monotone, scalar-monotone and preserves global minima.
\end{definition}

Before we proceed, let us have a brief discussion on the above properties:\medskip
\newline Property \textrm{(D1)} states that there is \textit{no descent at global minima}; thus
$T[f](x)=0$ holds at every $x\in\argmin f$.\smallskip\newline Property \textrm{(D2)}
expresses the fact that the \textit{amount of descent} of $f$ at a point $x$ depends only on the sublevel set $\lbrack f\leq f(x)]$ and  is captured by the function $z\mapsto(f(x)-f(z))_{+}$. Accordingly, for a
fixed $x\in X$, the relation
\begin{equation*}\label{eq:PreorderDef}
g\preceq_{x}f\quad\underset{\text{def}}{\iff}\quad\forall z\in X:\quad(g(x)-g(z))_{+}\leq
(f(x)-f(z))_{+}
\end{equation*}
is a preorder relation on $\mathcal{F}$ which roughly reads as follows: ``$f$ has more descent than $g$ at $x$". Under this terminology, (\textrm{D2}) requires the mapping $\mathcal{F}\ni f\mapsto T[f](x)$ to be monotone
with respect to $\preceq_{x}$. \smallskip\newline Notice further that
(\ref{eq:ad}) yields the following: If $g\preceq_x f$, then
\begin{equation}\label{eq:ad2}
z\notin\lbrack f< f(x)]\quad\Longrightarrow\quad g(z)\geq g(x)\quad\text{(that
is, } \,z\notin\lbrack g<g(x)]\text{)}. 
\end{equation}
This means that $g\preceq_x f$ implies that $[g< g(x)]\subset [f< f(x)]$. Finally, scalar-monotonicity in {\rm (D3)} expresses the fact that the descent of the function $g=rf$ should be larger than the one of~$f$, when $r>1$. \smallskip\newline    
In conclusion, the above axioms (D1)--(D3) are natural requirements for an abstract notion of descent of a function $f$ at a point $x$. 
The following proposition reveals further properties that can be derived from the axioms of Definition~\ref{sec03:def:ModulusOfDescent}.

\begin{proposition}[Properties of descent moduli]\label{sec03:prop:EquivalentProperties} Let $\mathcal{F}\subset\mathcal{C}(X)$ (as before) and $T:\mathcal{F}\to (\Rex_+)^X$ be an operator. Then: \smallskip\newline
$\mathrm(a).$ {\rm (one-step descent property)} $T$ is monotone if and only if for every $f,g\in \mathcal{F}$  and $x\in X$
	\begin{equation}\label{sec03:eq:one-stepDescent}
		T[f](x) > T[g](x)\quad \implies \quad \exists z\in [f<f(x)]:\,\, f(x) - f(z) > g(x) - g(z).  
	\end{equation}
$\mathrm(b).$ {\rm (translation-invariance)} If $T$ is a descent modulus for $\mathcal{F}$, then for every $c\in \mathbb{R}$ and $f\in \mathcal{F}$ we have: $$T[f] = T[f+c].$$
$\mathrm(c).$ {\rm (strict monotonicity)} Let $T$ be monotone at $x\in X$. Then the following are equivalent:\smallskip\newline
-- $\mathrm(c_1).$ $T$ is scalar-monotone at $x$.\smallskip\newline
-- $\mathrm(c_2).$ For every $f,g\in \mathcal{F}$ with $T[f](x)>0$, $T[g](x)<+\infty$ and $[g\leq g(x)]\subset [f\leq f(x)]$,
the implication 
$$\exists \delta>0:\, \forall z \in [g\leq g(x)]\,\, \implies\,\, f(x) - f(z)\geq (1+\delta)(g(x) - g(z)),$$
yields $$T[f](x) > T[g](x).$$
-- $\mathrm(c_3).$ For every $f\in \mathcal{F}$, $x\in X$ and $r\in (1,+\infty)$ such that $0<T[f](x)$ and $T[rf](x) <+\infty$, the mapping 
$$ [0,r-1]\ni \delta\,\,\mapsto \,\, T[(1+\delta)f](x)$$ is strictly increasing.
\end{proposition}

\begin{proof} Let us show the above properties separately:  \smallskip\newline
$\mathrm(a).$ (\textit{sufficiency}) Reasoning by absurd, assume $T$ verifies the one-step property but it is not monotone. Then, there exist $f,g\in\mathcal{F}$ and $x\in X$  with $(f(x)-f(\cdot))_+\geq (g(x)-g(\cdot))_+$ but with $T[f](x)<T[g](x)$. By the one-step descent property \eqref{sec03:eq:one-stepDescent}, there exists $z\in [g< g(x)]$ such that
\[
g(x) - g(z) > f(x) - f(z).
\]
However, the inequality $(f(x)-f(\cdot))_+\geq (g(x)-g(\cdot))_+$ yields that  $[g<g(x)] \subset [f\leq f(x)]$, and so the above inequality is a direct contradiction. \smallskip

(\textit{necessity}) Assume that $T$ is monotone but the one-step descent property does not hold. Then, there exist $f,g\in\mathcal{F}$ and $x\in X$  with $T[f](x) > T[g](x)$ such that for all $z\in X$ we either have $f(x)\leq f(z)$ or $f(x)-f(z)\leq g(x)-g(z)$.  It is not hard to see that for every $z\in X$ one has that
\[
(f(x)-f(z))_+ = \begin{cases}
	\, f(x)-f(z) \leq g(x)-g(z)\,,\quad&\text{ if }f(x)>f(z)\medskip \\
	\phantom{david salasios}0\,, &\text{ otherwise.}
\end{cases}
\]
Thus, in any case, we get that $(f(x)-f(z))_+\leq (g(x)-g(z))_+$. Then, monotonicity yields that $T[f](x)\leq T[g](x)$, which is a contradiction. \medskip\newline

$\mathrm(b).$ We show that for every $f\in \mathcal{F}$ and $c\in \mathbb{R}$, we have $T[f] = T[f+c]$. \smallskip\newline
Notice that $\left[ (f(x)+c) - (f(z)+c)\right]_+ \geq  \left[f(x)-f(z)\right]_+$ holds trivially for all $x,z\in X$. By monotonicity we deduce that $T[f]\leq T[f+c]$. Replacing now $f$ by $f' = f+c$ and respectively, $c$ by $c' = -c$, we obtain equality. \smallskip\newline

$\mathrm(c).$ Let us show first that $(c_1)\Rightarrow (c_2)$:\smallskip\newline Reasoning by absurd, assume that there exist $f,g\in\mathcal{F}$ and $\delta>0$ satisfying the hypotheses of the statement and $x\in X$ with $0 < T[f](x)\leq T[g](x)< +\infty$. Then for all $z\in X$ it holds:
\[
(f(x) - f(z))_+\geq ((1+\delta)g(x) - (1+\delta)g(z))_+\,. 
\]
By monotonicity, we deduce that $T[(1+\delta)g] \leq T[f]$. Further, using scalar-monotonicity, we get
\[
0 < T[g](x)< +\infty\, \implies \, T[g](x) < T[(1+\delta) g](x) \leq T[f](x) \leq T[g](x),
\]
which is a contradiction.\smallskip\newline
Let us now show that $(c_2)\Rightarrow (c_3)$: \smallskip\newline
Let $f\in \mathcal{F}$, $x\in X$ and $r>1$ such that $0<T[f](x), T[rf](x)<+\infty$. Fix $\delta_1,\delta_2 \in [0,r-1]$ with $\delta_1<\delta_2$. We set $ g = (1+\delta_1)f$ and $h = (1+\delta_2)f$. Monotonicity yields that $0\leq T[g](x)$ and $T[h](x)<+\infty$. Then, setting $$\delta = \frac{1+\delta_2}{1+\delta_1} - 1$$ we have that for all $y\in X$, $[g\leq g(y)] = [h\leq h(y)]$,  and for all $y,z\in X$:
	\[
	h(y) - h(z) = (1+\delta)(g(y) - g(z)).
	\]
Thus, (c2) yields that $T[h](x)>T[g](x)$. \smallskip\newline

Let us finally establish that $(c_3)\Rightarrow (c_1)$. \smallskip\newline
To this end,  let $f\in \mathcal{F}$, $r>1$ and $x\in X$ such that $0<T[f](x)<+\infty$. We need to show that $T[f](x) < T[rf](x)$.
This holds trivially if $T[rf](x) = +\infty$, therefore we can assume that $T[rf](x) <+\infty$. Since $T$ is monotone, we have that
$T[rf](x) \geq T[f](x)$ already, and in particular, $T[rf](x)>0$. Then, by hypothesis, we have that the mapping $$[0,r-1]\ni \delta \mapsto T[(1+\delta)f](x)$$ is strictly increasing, which leads us directly to the desired inequality.
\end{proof}


\subsection{Determination in topological spaces}\label{sec03-02:GeneralDetermination}


Let $(X,\tau )$ be a topological space and let $T$ be a descent modulus for $\mathcal{K}(X)$. Let us define the following
equivalent relation on the class $\mathcal{K}(X)$ of continuous coercive
functions: we say that the functions $f,g\in \mathcal{K}(X)$ are equivalent (and we denote $f\thicksim g$) if they have the same $T$-critical set and they are equal there.\\
In other words:
\begin{equation*}
f\thicksim g\qquad \Longleftrightarrow \qquad \mathcal{Z}_T(f)=\mathcal{Z}_T(g)=S\quad \text{and}\quad f|_{S}=g|_{S}.
\end{equation*}
In this section, borrowing from techniques developed in \cite{DS2022}, we
show that properties (D1)--(D3) of the descent modulus (\textit{cf.}
Definition~\ref{sec03:def:ModulusOfDescent}) are sufficient to guarantee
that the mapping $f\mapsto T[f]$ is injective on $\mathcal{K}(X)$, modulo
the above equivalent relation. Therefore, according to our
terminology, the descent modulus \textit{determines} the class $\mathcal{K}
(X)$. At this stage, let us also outline the topological nature of this result: no linear or
metric structure is required. \smallskip \newline
The results of this section will be stated in a slightly more general
framework. We assume, similarly to the previous section, that $\mathcal{F}\subset\mathcal{K}(X)$ 
is a translation cone.\smallskip \newline
We start with the following lemma.

\begin{lemma}[strict domination of descent modulus]
\label{sec03:lemma:strictComparison} Let $T$ be a descent modulus for the
class~$\mathcal{F}$. Let $f,g\in \dom(T)$
such that 
\begin{equation*}
\forall x\in X\setminus \mathcal{Z}_T(f),\quad T[f](x) > T[g](x).
\end{equation*}
Then, for all $x\in X$, we have that 
\begin{equation*}
f(x) \geq g(x) + \mu(x),
\end{equation*}
where 
\begin{equation*}
 \mu(x) := \inf\{(f-g)(z)\ :\ z\in [f\leq f(x)]\cap \mathcal{Z}_T(f)\}\in 
\mathbb{R}\cup \{-\infty\}.
\end{equation*}
\end{lemma}

\begin{proof}
Since the set of global minimizers of every function $f\in\mathcal{F}$ is nonempty and the abstract descent modulus $T$ preserves global minima, we deduce that $\mathcal{Z}_T(f)\neq \emptyset$ and consequently, $ \mu(x) <+\infty$. Let us assume, towards a contradiction, that there exists $x\in X$ such that $f(x)<g(x) + \mu(x)$. Then, clearly $ \mu(x)>-\infty$ which readily yields $x\in X\setminus \mathcal{Z}_T(f)$. 
Therefore, by assumption, $T[f](x) > T[g](x)$. Applying the one-step descent property~\eqref{sec03:eq:one-stepDescent} of~$T$, we infer that there exists $z_0\in X$ such that
\[
f(z_0)<f(x)\quad\text{and}\quad(g-f)(z_0) = c > (g-f)(x)> - \mu(x).
\]
Since $z_0$ is not a $T$-critical point, we can repeat the above argument to obtain $z_1\notin\mathcal{Z}_{T}$ such that  
$f(z_1)<f(z_0)$ and $(g-f)(z_1) > c = (g-f)(z_0)$. Following the strategy of \cite[Proposition~2.2]{DS2022}, we construct (by means of a transfinite induction over the ordinals) a generalized sequence $\{z_{\alpha}\}_{\alpha}\subset [f\leq f(z_0)]$ such that $\{f(z_{\alpha})\}_a$ is decreasing  and $\{(g-f)(z_\alpha)\}_{\alpha}$ is increasing: \smallskip\newline
-- If $\alpha = \beta+1$ is a successor ordinal then, since $z_{\beta}\notin \mathcal{Z}_T(f)$ and $g(z_{\beta}) \geq f(z_{\beta}) +c$, the one-step descent property~\eqref{sec03:eq:one-stepDescent} yields $z_{\beta+1}$ such that
    \[
    f(z_{\beta +1})< f(z_{\beta}) \leq f(z_0)\quad\text{ and }\quad(g-f)(z_{\beta+1})> (g-f)(z_\beta)\geq c.
    \]
-- If $\alpha$ is a limite-ordinal and $\{z_{\beta}\}_{\beta <\alpha}\subset [f\leq f(z_0)]$ is defined accordingly, then since the sublevel set $[f\leq f(z_0)]$ is compact, the $\omega$-limit set
    \[
    A = \bigcap_{\beta < \alpha} \overline{\{ z_{\eta}\ :\ \beta\leq\eta < \alpha \}},
    \]
is nonempty. Pick any $z_{\alpha}\in A$. Clearly, $z_{\alpha}\in [f\leq f(z_0)]$, $f(z_{\alpha})\leq f(z_{\beta})$ for all $\beta\leq \alpha$, and by continuity
\[
    (g-f)(z_{\beta}) = \inf \{ (g-f)(z_{\eta})\ :\ \beta\leq\eta < \alpha \} \leq (g-f)(z_{\alpha}).
\]
Notice that the above construction never meets a $T$-critical point of $f$. Indeed, if $z_{\alpha}\in \mathcal{Z}_T(f)$ for some ordinal $\alpha$, then since $f(z_{\alpha})<f(x)$ we would have that
\begin{align*}
   -\mu(x) &= \sup\{ (g-f)(z)\ :\ z\in [f\leq f(x)]\cap \mathcal{Z}_T(f) \}\\
   &\geq (g-f)(z_{\alpha}) \geq c > -\mu(x), 
\end{align*}
which is a contradiction. Due to a cardinality obstruction, we necessarily deduce that $z_{\alpha}=z_{\beta}$ for some ordinals $\alpha$, $\beta$ with $\alpha>\beta$. This yields
\[
(g-f)(z_{\beta+1})>(g-f)(z_{\beta})= (g-f)(z_{\alpha})\geq (g-f)(z_{\beta+1}), 
\]
which is clearly a contradiction. The conclusion follows.
\end{proof}


\begin{theorem}[Comparison principle]
\label{sec03:thm:ComparisonPrinciple} Let $T$ be a descent modulus for $\mathcal{F}$ and let $f,g\in \dom(T)$ and $c\in \mathbb{R}$ such that

\begin{itemize}
\item[\textrm{(i).}] $T[f](x) \geq T[g](x)$, for all $x\in X$; and

\item[\textrm{(ii).}] $f(\bar x)\geq g(\bar x)+c$, for all $\bar x\in 
\mathcal{Z}_T(f)$.
\end{itemize}

Then, $f\geq g+c$.
\end{theorem}

\begin{proof}
Let $x \in X\setminus\mathcal{Z}_T(f)$ be arbitrarily chosen. Fix $\varepsilon>0$, set $f_{\varepsilon
} = (1+\varepsilon)f$ and notice that monotonicity of $T$ yields that $\mathcal{Z}_T(f_{\varepsilon})\subset \mathcal{Z}_T(f)$. Let $z\in X\setminus \mathcal{Z}_T(f_{\varepsilon})$. We have two cases:
\begin{itemize}
    \item[Case 1:] $z\in X\setminus\mathcal{Z}_T(f)$. Then scalar-monotonicity of $T$ yields
\[
T[f_{\varepsilon}](z) = T[(1+\varepsilon)f](z) > T[f](z) \geq T[g](z).
\]
\item[Case 2:] $z\in \mathcal{Z}_T(f)\setminus\mathcal{Z}_T(f_{\varepsilon})$. Then $T[f_{\varepsilon}](z)>0 = T[f](z)\geq T[g](z)$.
\end{itemize}
In both cases, $T[f_{\varepsilon}](z)> T[g](z)$. Thus, by Lemma~\ref{sec03:lemma:strictComparison}, we have that
\begin{align*}
    f_{\varepsilon}(x) &\geq  g(x) + \inf\{ (f_{\varepsilon}-g)(z)\ :\ z\in \mathcal{Z}_T(f_{\varepsilon})\cap [f_{\varepsilon}\leq f_{\varepsilon}(x)] \}\\
    &\geq g(x) + \inf\{ (f_{\varepsilon}-g)(z)\ :\ z\in \mathcal{Z}_T(f)\cap [f\leq f(x)] \}\\
    &\geq g(x) + c + \varepsilon \inf\{ f(z) :\ z\in \mathcal{Z}_T(f)\cap [f\leq f(x)]\}\\
    &\geq g(x) + c + \varepsilon\min f.
\end{align*}
Finally, by taking $\varepsilon\to 0$, we deduce that $f(x)\geq g(x) + c$. 
 The proof is complete.
\end{proof}

Applying twice Theorem~\ref{sec03:thm:ComparisonPrinciple}, we deduce easily
the following determination result.

\begin{theorem}[Determination of continuous coercive functions]
\label{sec03:thm:Determination} Let $T$ be a descent modulus for a translation
cone $\mathcal{F}$ of $\mathcal{K}(X)$. Let $f,g\in \dom(T)$ and $c\in 
\mathbb{R}$ be such that

\begin{itemize}
\item[(i).] $T[f](x) = T[g](x)$ for all $x\in X$ (whence $\mathcal{Z}_T(f)=
\mathcal{Z}_T(g)$); and

\item[(ii).] $f(x) = g(x) + c$, for all $x\in \mathcal{Z}_T(f)$.
\end{itemize}

Then, $f = g + c$.
\end{theorem}

\begin{remark} A descent modulus $T$ for a class $\mathcal{F}$ is meant to assign a quantified measure of descent at every point of $f\in\mathcal{F}$. This quantity is also allowed to be infinite at some points of some functions and whenever this happens the determination result cannot apply. Therefore, $T$ does not determine the whole class $\mathcal{F}$, but instead only functions in $\dom(T)\subset \mathcal{F}$.
\end{remark}


\subsection{Stability properties of descent moduli and examples\label{sec03-03:Stability}}

The metric slope (used in \cite{AGS2008}, \cite{GMT1980} \textit{e.g.}) is a
natural instance of abstract descent modulus and the results of the previous
section can be seen as a minimal axiomatic presentation of the slope determination
result given in \cite{DS2022}. In this section, we show that the axiomatic
descent modules also captures the notion of global slope (used in \cite{TZ2022}) as well as several natural adaptations of the notion of slope to
topological spaces, emancipating from the metric framework. \smallskip
\newline
Throughout this section, $\mathcal{F}$ will denote a translation cone of $\mathcal{C}(X)$. 

\begin{proposition}[$m$-slope]
\label{prop-m-slope} Let $m:X\times X\to \mathbb{R}_+$ be a mapping
satisfying: 
\begin{equation*}
m(x,y) = 0 \quad\iff \quad x= y \qquad\text{(separation axiom)}.
\end{equation*}
Let further $\mathcal{D} = \{ \mathcal{D}_x \}_{ x\in X}$ be a family of subsets of $X$
satisfying $x\in \mathcal{D}_x$ for every $x\in X$. Then, the $m$-slope 
\begin{equation*}
s_f(x) := \left\{
\begin{array}{cl}
\displaystyle\limsup_{y\to x} \Delta_f^+(x,y) & \text{ if }x\text{ is not
isolated,} \\ 
0 & \text{ otherwise,}
\end{array}
\right.
\end{equation*}
and the semiglobal $(\mathcal{D},m)$-slope 
\begin{equation}\label{eq:D-global}
\mathscr{G}_{\mathcal{D}}[f](x) = \sup_{y\in \mathcal{D}_x} \Delta_f^+(x,y)
\end{equation}
are moduli of descent for the class $\mathcal{F}$, where 
\begin{equation}
\Delta_f^+(x,y) = \left\{
\begin{array}{cc}
\frac{(f(x) - f(y))_+}{m(x,y)} & \text{ if }y\neq x, \\ 
&  \\ 
0 & \text{ if }y=x.
\end{array}
\right.
\end{equation}
\end{proposition}

\begin{proof}
Let us show that the above operators of (local) $m$-slope and (semiglobal) $(\mathcal{D},m)$-slope satisfy axioms (D1)--(D3) of Definition~\ref{sec03:def:ModulusOfDescent}. It is straightforward to see that (D1) (preservation of global minima) is fulfilled. Axiom (D3) (scalar monotonicity) is also fulfilled, since for every$f\in \mathcal{F}$ and $r>0$ we have 
\[
s_{rf}(x) = r s_f(x)\quad\text{ and }\quad \mathscr{G}_{\mathcal{D}}[rf](x) = r\mathscr{G}_{\mathcal{D}}[f](x).
\]
It remains to show that both operators also satisfy axiom (D2) (Monotonicity). To this end, let $f,g\in \mathcal{F}$ such that
\[
(f(x) - f(z))_+ \geq (g(x) - g(z))_+.
\]
Then for every $y\in X$ we have $\Delta_f^+(x,y)\geq \Delta_g^+(x,y)$, which readily yields that $s_f(x) \geq s_g(x)$ and $\mathscr{G}_{\mathcal{D}}[f](x)\geq \mathscr{G}_{\mathcal{D}}[g](x)$. The proof is complete.
\end{proof}

\begin{remark}
When $(X,\tau)$ is a metric space and $m$ is the distance function, the $m$-slope $s_f(x)$ coincides with the usual metric slope $|\nabla f |(x)$ and
the main result of~\cite{DS2022} follows directly from Theorem~\ref{sec03:thm:Determination}. Taking now $\mathcal{D}_x=X$ for all $x\in X$, the semiglobal $(\mathcal{D},m)$-slope $\mathscr{G}_{\mathcal{D}}[f](x)$ coincides with the global slope $\globalSlope[f](x)$ (see, e.g., \cite[Definition~1.2.4]{AGS2008}) which was used in~\cite{TZ2022}.
\end{remark}

Notice that the semiglobal slope $\mathscr{G}_{\mathcal{D}}[f]$ is intrinsically different from the
metric slope (or the norm of the gradient $\|\nabla f\|$ in the
differentiable case), which already reveals that Definition~\ref{sec03:def:ModulusOfDescent} represents a much more general setting. The
next proposition shows that we can go even further.

\begin{proposition}[Constructing descent moduli]
\label{sec03:prop:StabilityModuli} $\mathrm{(i).}$ Let $T_1$, $T_2$ be
descent moduli for the class~$\mathcal{F}$. Then $T_1+ T_2$ is also a
descent modulus for $\mathcal{F}$, where 
\begin{equation*}
(T_1+ T_2)[f](x):= T_1[f](x) + T_2[f](x), \quad \text{for all }\, f\in 
\mathcal{F} \,\, \text{and }\, x\in \dom{f}.
\end{equation*}
$\mathrm{(ii)}.$ Let $T$ be a descent modulus for $\mathcal{F}$ and let $\phi:\mathbb{R}_+\to\mathbb{R}_+$ be a strictly increasing function with $\phi(0) = 0$ and $\limsup_{t\to+\infty}\phi(t) = +\infty$. Then 
\begin{equation*}
(\phi T)[f](x):=(\phi\circ T[f])(x), \quad \text{for all }\, f\in \mathcal{F}
\,\, \text{and }\, x\in \dom{f},
\end{equation*}
is also a descent modulus for $\mathcal{F}$, under the convention $\phi(+\infty) = \limsup_{t\to+\infty}\phi(t)=+\infty$. \smallskip\newline
In particular, $r T$, $r\geq 0$ is a descent modulus for $\mathcal{F}$,
where 
\begin{equation*}
(r T)[f](x):= r\cdot T[f](x),
\end{equation*}
under the convention $r\!\cdot\!
(+\infty) = +\infty$ for $r>0$, and $0\!\cdot\!
(+\infty) = 0$.
\end{proposition}

\begin{proof}
Let $T_1$, $T_2$, $T$ and $\phi$ as in the statements (i) and (ii). We show that axioms (D1)--(D3) of Definition~\ref{sec03:def:ModulusOfDescent} are fulfilled:\smallskip\newline
-- (D1) (\textit{Preservation of global minima}) Let $f\in \mathcal{F}$ and $x\in \argmin f$. Then $$T_1[f](x) = T_2[f](x) = T[f](x) = 0$$ and consequently $(T_1 + T_2)[f](x) = 0$\, and \,$\phi(T[f](x)) = \phi(0) = 0.$  Therefore, $T_1+T_2$ and $\phi T$ preserve global minima. \medskip\newline
-- (D2) (\textit{Monotonicity}) Let $f,g\in\mathcal{F}$ and $x\in X$ such that
\[
(f(x) - f(z))_+ \geq (g(x) - g(z))_+,\quad \forall x\in X.
\]
Then, since $T_1$ and $T_2$ are monotone, we have that
\[
(T_1+T_2)[f](x) = T_1[f](x)+T_2[f](x) \geq T_1[g](x) + T_2[g](x) = (T_1+T_2)[g](x).
\]
Similarly, since $T$ is monotone and $\phi$ is non-decreasing, we get that
\[
(\phi T)[f](x) = \phi(T[f](x))\geq \phi(T[g](x)) = (\phi T)[g](x).
\]
Thus, $T_1+T_2$ and $\phi T$ are monotone. \medskip\newline
-- (D3) (\textit{Scalar monotonicity}) Let $f\in \mathcal{F}$, $x\in X$ and $r>1$ and assume $0 < (T_1+T_2)[f](x) < +\infty$. Up to a mutual change of $T_1$ and $T_2$, we many assume $0<T_1[f](x) < +\infty$. Then, using the scalar monotonicity of $T_1$ and the monotonicity of $T_2$, we deduce
    \begin{align*}
    (T_1+T_2)[rf](x) &= T_1[rf](x) + T_2[rf](x) \, > \, T_1[f](x) + T_2[rf](x)\smallskip \\
    &\geq  T_1[f](x) + T_2[f](x) = (T_1+T_2)[f](x).
    \end{align*}
    Thus, $(T_1+T_2)$ is scalar-monotone. \smallskip\newline
 Let us now assume $0 < (\phi T)[f](x) < +\infty$. Since $\phi(0) = 0$ and $\phi(+\infty) = +\infty$, we obtain again $0<T[f](x)<+\infty$. Thus, $T[rf](x)>T[f](x)$ and 
    \[
    (\phi T)[rf](x) = \phi(T[rf](x))>\phi(T[f](x))=(\phi T)[rf](x),
    \]
yielding that $(\phi T)$ is scalar-monotone.
We conclude that both $(T_1+T_2)$ and $(\phi T)$ are descent moduli for $\mathcal{F}$.
\end{proof}

Notice that the family of descent moduli for the class $\mathcal{F}$ has the
structure of a convex cone (i.e., it is a cone closed for the sum), with the sum and the scalar multiplication being
defined as in Proposition~\ref{sec03:prop:StabilityModuli}. \smallskip
\newline
The following proposition provides other types of operations, based on 
\textit{truncations}, that preserve descent moduli.

\begin{proposition}[Truncated descents]
\label{prop:trunc} Let $T$ be a descent modulus for the class $\mathcal{F}
$. Then: \smallskip\newline
(i). For every $\varepsilon>0$, the operator $T_{\varepsilon}$ given by 
\begin{equation*}
T_{\varepsilon}[f](x) = \left\{
\begin{array}{cl}
T[f](x), & \text{ if }f(x) > \inf f + \varepsilon \smallskip \\ 
0, & \text{ otherwise,}
\end{array}
\right.
\end{equation*}
is a descent modulus for $\mathcal{F}$.\smallskip\newline
(ii). For every $K\subset X$, the operator $T\big|_K$ given by 
\begin{equation*}
T\big|_K[f](x) = \left\{
\begin{array}{cl}
T[f](x), & \text{ if }x\in K \smallskip \\ 
0, & \text{ otherwise,}
\end{array}
\right.
\end{equation*}
is a descent modulus for $\mathcal{F}$.
\end{proposition}

\begin{proof}
Let $T$, $\varepsilon>0$ and $K\subset X$ as in the statement of the proposition.  We will show that the operators $T_{\varepsilon}$ and $T\big|_{K}$ satisfy properties (D1)--(D3) of Definition~\ref{sec03:def:ModulusOfDescent}. Notice that for every $f\in \mathcal{F}$ and $x\in X$ we have $T[f](x)\geq T_{\varepsilon}[f](x)$ and $T[f](x) \geq T\big|_K[f](x)$. Therefore, if $T[f](x) = 0$, the above readily yields $T_{\varepsilon}[f](x) =\big|_{K}(x) = 0$, and (D1) holds trivially.\smallskip\newline
Let us now prove (D2). To this end, Let $f,g\in\mathcal{F}$ and $x\in X$ such that
\[
(f(x) - f(z))_+ \geq (g(x) - g(z))_+,\quad \forall x\in X.
\]
Let us first deal with $T_{\varepsilon}$: if $f(x) > \inf f + \varepsilon$, then $T_{\varepsilon}[f](x) = T[f](x) \geq T[g](x) \geq T_{\varepsilon}[g](x)$. On the other hand,  if $f(x) \leq \inf f + \varepsilon$, then  $[f(x) - f(z)]_+\leq \varepsilon$ for all $z\in X$, whence $g(x) \leq \inf g + \varepsilon$ and $T_{\varepsilon}[f](x) = T_{\varepsilon}[g](x) = 0$ (by definition of $T_{\varepsilon}$). We conclude that  $T_{\varepsilon}[f](x) \geq T_{\varepsilon}[g](x)$. \smallskip\newline
Let us now deal with $T\big|_K$: If $x\in K$, then $T\big|_K[f](x) = T[f](x) \geq T[g](x) = T\big|_K[g](x)$, while if $x\in X\setminus K$, then $T\big|_K[f](x) = T\big|_K[g](x) = 0$. In both cases $T\big|_K[f](x) \geq T\big|_K[g](x)$.\medskip\newline
It remains to prove (D3). Let $f\in \mathcal{F}$, $x\in X$ and $r>1$. If $\inf f=-\infty$, then $T_{\varepsilon}[f]=T[f]$ and the result follows. Therefore, we may assume $\inf f>-\infty$ and $0<T_{\varepsilon}[f](x) < + \infty$. This yields $f(x)> \inf f + \varepsilon$ and consequently, $T_{\varepsilon}[f](x) = T[f](x)$. Noting that $$rf(x) > r(\inf f + \varepsilon) > \inf rf + \varepsilon,$$ we conclude that $T_{\varepsilon}[rf](x) = T[rf](x)$, as well. Then, since $T_{\varepsilon}[f](x) = T[f](x)<T[rf](x) = T_{\varepsilon}[rf](x)$, we conclude that $T_{\varepsilon}$ is scalar-monotone.\smallskip\newline
Let us now assume $0<T\big|_K[f](x) < + \infty$. This yields in particular that $x\in K$ and so $T\big|_K[f](x) = T[f](x)$ and $T\big|_K[rf](x) = T[rf](x)$. Then, since $T\big|_K[f](x) = T[f](x)<T[rf](x) = T\big|_K[rf](x)$, we conclude that $T\big|_K$ is scalar-monotone.
\end{proof}

The last stability property that we study is the pointwise limit. In
general, this operation does not preserve moduli of descent, since
scalar-monotonicity can be lost in the limit process, as the following
example reveals.

\begin{example}[Axiom (D3) is not preserved under pointwise limits]
\label{sec03-03:ex:LimitNotModulus} Let $X=\mathbb{R}^n$ and consider the
class $\mathcal{F}= \mathcal{C}^1(\mathbb{R}^n)$ of $\mathcal{C}^1$-smooth
functions. Let us further consider the sequence of descent moduli 
\begin{equation*}
T_{n}[f](x) = \sqrt[n]{\|\nabla f(x)\|}, \quad n\in \mathbb{N},
\end{equation*}
and its pointwise limit operator: 
\begin{equation*}
T[f](x) = \lim_{n\to \infty} T_n[f](x) = 
\begin{cases}
\phantom{jo}0\,,\quad & \text{ if }\nabla f(x) = 0, \\ 
\phantom{jo}1\,, & \text{ otherwise.}
\end{cases}
\end{equation*}
The operator $T$ preserves global minima and is monotone. However, it is not
scalar-monotone (and it clearly fails to determine functions in the sense of
Theorem~\ref{sec03:thm:Determination}.)
$\hfill\Diamond$ 
\end{example}

The following definition introduces a large subclass of abstract descent
moduli which provides a remedy to the above situation.

\begin{definition}[Homogeneous descent moduli]
\label{def-homog} Let $\mathcal{F}\subset\mathcal{C}(X)$
be a translation cone, and let $p\in (0,+\infty )$. An
operator $T:\mathcal{F}\rightarrow (\overline{\mathbb{R}}_{+})^{X}$ is said
to be

\begin{itemize}
\item[$(i).$] \textit{$p$--homogeneous} if $T[rf](x) = r^p\,T[f](x)$, for
every $f\in \mathcal{F}$ and $r>0$.

\item[$(ii).$] \textit{$p$--superhomogeneous} if $T[rf](x) \geq r^p\,T[f](x)$
, for every $f\in \mathcal{F}$ and $r>0$.
\end{itemize}
\end{definition}

Clearly all $p$--homogeneous and all $p$--superhomogeneous operator are also
scalar-monotone. The interest of this class is that every operator $T$ which
is defined as a pointwise limit of a sequence of $p$-(super)homogeneous
descent moduli $\{T_n\}_{n\in\mathbb{N}}$, that is, 
\begin{equation*}
T[f](x)=\lim_{n\to +\infty}T_n[f](x),\qquad\text{for all } f\in\mathcal{F}
\,\,\text{and }\, x\in \dom(f),
\end{equation*}
is itself a $p$--(super)homogeneous descent modulus. In other words, axiom
(D3) (scalar-monotonicity) is preserved in this context. One can also
observe that up to a composition with the strictly increasing function $\varphi(t):=t^{1/p}$, $p$--(super)homogenicity reduces to $1$
--(super)homogenicity.

\begin{proposition}
Let $(\Lambda, \preccurlyeq)$ be a directed set, $p\in (0,+\infty)$ and $(T_{\alpha})_{\alpha\in\Lambda}$ be a generalized sequence of $p$
--(super)homogeneous descent moduli for the class $\mathcal{F}$. Then the
following operators, defined for every $f\in \mathcal{F}$ and $x\in \dom(f)$
, are descent moduli for the class $\mathcal{F}$:\medskip\newline
$(i).$ $\left(\underset{\alpha\in\Lambda}{\limsup } \,T_{\alpha}
\right)[f](x) := \underset{\alpha\in\Lambda}{\limsup }\, T_{\alpha}[f](x)$;
\smallskip\newline
$(ii).$ $\left(\underset{\alpha\in\Lambda}{\sup }\, T_{\alpha}\right)[f](x)
:= \underset{\alpha\in\Lambda}{\sup } \,T_{\alpha}[f](x)$;\smallskip\newline
$(iii).$ $\left(\underset{\alpha\in\Lambda}{\liminf }\,
T_{\alpha}\right)[f](x) := \underset{\alpha\in\Lambda}{\liminf }\,
T_{\alpha}[f](x)$ ; \smallskip\newline
$(iv).$ $\left(\underset{\alpha\in\Lambda}{\inf }\, T_{\alpha}\right)[f](x)
:= \underset{\alpha\in\Lambda}{\inf}\, T_{\alpha}[f](x)$.
\end{proposition}

\medskip

\begin{proof} Let us verify that $T := \limsup_{\alpha} T_{\alpha}$ satisfies axioms (D1)--(D3) of Definition~\ref{sec03:def:ModulusOfDescent}. (A~si\-mi\-lar reasoning will apply to the other three operators.)\smallskip\newline 
-- (D1) (\textit{Preservation of global minima}) Choose $f\in \mathcal{F}$ and $x\in\argmin f$. Then, $T_{\alpha}[f](x) = 0$ for all $\alpha\in \Lambda$ and so $T[f](x) = 0$. Thus, $T$ preserves global minima. \smallskip\newline 
-- (D2) (\textit{Monotonicity}). Let $f,g\in \mathcal{F}$ and $x\in X$ be such that $ (f(x) - f(z))_+\geq (g(x) - g(z))_+,$ for all $z\in X$. Then, $T_{\alpha}[f](x)\geq T_{\alpha}[g](x)$ for each $\alpha\in \Lambda$. Thus, $T[f](x)\geq T[g](x)$ as well, showing that $T$ is monotone \smallskip\newline 
-- (D3) (\textit{Scalar-monotonicity}): Let $f\in \mathcal{F}$ and $x\in X$. We readily deduce from $p$-superhomogeneity that $T[rf](x) = \limsup_{\alpha} T_{\alpha}[rf](x)\,\geq\, r^p\, \limsup_{\alpha} T_{\alpha}[f](x) = \, r^p\,T[f](x)$. It follows that $T$ is also $p$--superhomogeneous, therefore, in particular, scalar-monotone. \smallskip\newline 
The proof is complete. 
\end{proof}

\subsection{Slope-like operators that are not descent moduli}

We finish this section by discussing two examples in the literature that have being introduced as ``slope operators'' on a metric space $(X,d)$, but fail to verify Definition~\ref{sec03:def:ModulusOfDescent} of descent modulus.

The first concept is the so-called \textit{weak slope}, introduced in \cite{DM1994Critical,CDM1993Deformation}. For a continuous function $f:X\to \mathbb{R}$, the weak slope at a point $x\in X$, denoted by $|df|(x)$, is defined as the supremum of $\sigma\in\mathbb{R}_+$ such that there exist $\delta>0$ and a continuous map $\mathcal{H}:[0,\delta]\times B(x,\delta)\to X$ such that 
\begin{equation}\label{eq:CondWeakSlope}
  \forall s\in [0,\delta],\,\forall y\in B(x,\delta),\quad d(\mathcal{H}(s,y),y)\leq s\,\, \text{ and } \, f(\mathcal{H}(s,y)) \leq f(y) - \sigma s.
\end{equation}

Notice that $|df(x)|\geq \sigma$ whenever it is possible to find a continuous deformation $\mathcal{H}$ over a neighborhood of $x$, such that the descent of $f$ through that deformation is at least $\sigma$ for every point $y$ over which $\mathcal{H}$ is acting. Thus, one might interpret the weak-slope as the slowest descent around $x$. This concept has been largely studied in the setting of nonsmooth variational analysis and critical point theory.

The second concept is the \textit{limiting slope} (see, e.g. \cite[Definition 8.4]{Ioffe2017Variational}), which is defined as the lower semicontinuous envelope (or closure) of the strong slope $|\nabla f|$. That is, for a lower semicontinuous function $f:X\to \mathbb{R}$ and a point $x\in X$, the limiting slope of $f$ at $x$ is given by
\begin{equation}
\overline{|\nabla f|}(x) := 
\lim_{\varepsilon \to 0}\inf\left\{|\nabla f |(y)\ :\ d(x,y)\leq \varepsilon,\text{ and }f(y)\leq f(x)+\varepsilon\right\}.
\end{equation}
Since the slope that can be very ill-behaved, the limiting slope provides a regularized alternative. It is worth to mention that 
using this notion, Drusvyatskiy,  Ioffe and Lewis were able to deal with the long-standing problem of existence of steepest descent curves \cite{DIL2015}.
\smallskip\newline
The following example shows that the weak slope and the limiting slope are not descent moduli for $\mathcal{K}(X)$, since they fail to determine coercive continuous functions even in the interval $[0,1]$.
\begin{example} Let  $\mathfrak{c}:[0,1]\to[0,1]$ be the well-known Cantor Staircase and let us consider the function $f:[0,1]\to \mathbb{R}$ given by $f(t) = \mathfrak{c}(t) + t$. By construction, it is not hard to see that $|\nabla f|(t) \in \{ 1,+\infty\}$ for every $t\in (0,1]$, that $|\nabla f|(0) = 0$ (since $0\in\argmin f$), and that the slope is $+\infty$ only in a subset of the Cantor set. Thus,
	\[
	\overline{|\nabla f|}(t) = \ind{(0,1]}(t):=\begin{cases}
		1,\quad&\text{ if }t\in(0,1]\\
		0,\quad&\text{ if }t = 0.
	\end{cases}
	\]
Similarly, we claim that $|df|(t)$ takes the same values as $\overline{|\nabla f|}(t)$. Clearly $|df|(0) = 0$ and $|df|(t)\geq 1$ for all $t\in (0,1]$. Now, fix $\bar{t} \in (0,1]$  and take any $\sigma>0$, $\delta>0$ and $\mathcal{H}$ satisfying \eqref{eq:CondWeakSlope}.  Since $f$ is strictly increasing, $\mathcal{H}(s,t)<t$ for every $t\in B(\bar{t},\delta)$ and every $s\in [0,\delta]$. In particular, $0<t-\mathcal{H}(s,t)= d(\mathcal{H}(s,t),t)\leq s$. Whence $t-s\leq \mathcal{H}(s,t)$ and consequently, $f(t-s)\leq f(\mathcal{H}(s,t))$.  Since the Cantor set is totally disconnected, there exists $t\in (\bar{t}-\delta,\bar{t})$ such that $|\nabla f|(t) = 1$. Then, 
\[
|\nabla f|(t) \geq \limsup_{s\to 0^+}\, \frac{f(t) - f(t-s)}{s}\,\geq \,\limsup_{s\to 0^+} \frac{f(t) - f(\mathcal{H}(s,t))}{s}\geq \sigma.
\]
Thus, $\sigma\leq 1$, which proves that $|df(\bar{t})|=1$. This proves the claim.
By taking $g:[0,1]\to \mathbb{R}$ given by $g(t) = t$, we get that $\overline{|\nabla g|}(t) = |dg|(t) = \ind{(0,1]}(t)$, and so, the conclusion of Theorem~\ref{sec03:thm:Determination} fails to hold for both the weak and the limiting slope. Since clearly both operators preserve global minima and are scalar-monotone (by homogeneity), we conclude that both operators fail to be monotone in the sense of Definition~\ref{sec03:def:ModulusOfDescent}.\hfill$\diamond$
\end{example}


\section{The paradigm of averaged descent} \label{sec04:Averaged}

It was shown in \cite[Theorem~3.8]{BCD2018} that two $\mathcal{C}^2$-smooth convex and bounded from below functions $f,g$ defined on a Hilbert space $\mathcal{H}$ are equal up to a constant, provided $\|\nabla f(x)\| = \|\nabla g(x)\|$, for all $x\in \mathcal{H}$. In other words, the operator:  
\begin{equation}\label{eq:gamma}
f\mapsto \Gamma[f] := \|\nabla f \|^2
\end{equation}
is injective, modulo the constant functions, on the class of $\mathcal{C}^2$-smooth convex and bounded from below functions.
Notice that the $\Gamma$-operator defined in \eqref{eq:gamma} (also known as \textit{carr\'e-du-champ} operator) is strongly related to the Wiener diffusion process, generated by the Laplacian operator. This hints towards a new important instance of descent modulus, namely the average descent, giving rise to a determination result of probabilistic nature. This will be developed in this section, in full generality.
\subsection{Extension of dispersion measures \label{sec04-01:extension}}
We first recall that for a $\mathcal{C}^1$-smooth function $f:\mathbb{R}^n\to\mathbb{R}$ the following formula holds: 
\begin{equation}\label{sec04:eq:IntegroDiff-formula}
\|\nabla f(x)\|^2 = \lim_{\varepsilon\to 0} \frac{n}{\mathcal{L}_n(B_n(x,\varepsilon))}\int_{B_n(x,\varepsilon)}\left[ \frac{f(x) - f(y)}{\|x-y\|} \right]^2 \mathcal{L}_n(dy),
\end{equation}
where, as mentioned in Section~\ref{sec02:Pre}, $\mathcal{L}_n$ stands for the usual Lebesgue measure on 
$\mathbb{R}^n$. The above formula is well-known and can be deduced from the following (also well-known) lemma, for which we provide a simple proof for completeness. 

\begin{lemma} For any $k\geq 1$, any $r>0$ and $V\in \mathbb{R}^k$ it holds:
	\begin{equation}
		\|V\|^{2}\,=\,\frac{k}{\mathcal{L}_{k}(B_{k}(0,r))}\int_{B_{k}(0,r)}
		\lt\langle V,\frac{u}{\|u\|}\rt\rangle ^{2}du.  \label{eq:jaime}
	\end{equation}
\end{lemma}

\begin{proof}
The proof is a  consequence of the invariance by rotations of the ball.
Consider $(e_i)_{i=1}^k$  the usual orthonormal basis of $\mathbb{R}^k$.
By symmetry, we can restrict to the case where $V = \| V\| \cdot e_1$, so that
\bq
\int_{B_{k}(0,r)}	\lt\langle V,\frac{u}{\|u\|}\rt\rangle ^{2}\, du&=&\| V\|^2\int_{B_{k}(0,r)}	\lt\langle e_1,\frac{u}{\|u\|}\rt\rangle ^{2}\,du\\
&=&\| V\|^2\int_{B_{k}(0,r)}	\lt\langle e_i,\frac{u}{\|u\|}\rt\rangle ^{2}\,du
\eq
for any $i\in\{1, ..., n\}$. We deduce
\bq
\int_{B_{k}(0,r)}	\lt\langle V,\frac{u}{\|u\|}\rt\rangle ^{2}\, du&=&\f{\| V\|^2}{k}\int_{B_{k}(0,r)}\sum_{i=1}^n \lt\langle e_i,\frac{u}{\|u\|}\rt\rangle ^{2}\,du\\
&=&\f{\| V\|^2}{k}\int_{B_{k}(0,r)}\lt \| \frac{u}{\|u\|} \rt\|^2\, du\\
&=&\f{\| V\|^2}{k}\int_{B_{k}(0,r)}1\, du\\
&=&\f{\| V\|^2\mathcal{L}_{k}(B_{k}(0,r))}{k}\eq
leading to the desired equality.

\end{proof}

Based on equation \eqref{sec04:eq:IntegroDiff-formula}, we propose an extension of the $\Gamma$-operator \eqref{eq:gamma}, that we call \textit{dispersion operator}, for functions defined on a topological space $(X,\tau)$. \smallskip\newline 
To this end, we consider the family $\beta = \{\beta_x\}_{x\in X}$ of neighborhood bases: $\beta_x$ is a neighborhood base at $x$ of the topology $\tau$. We further denote by $$\mu: X\times \mathcal{B}(X)\to \mathbb{R}_+$$ 
a mapping that associates to every $x\in X$, a locally finite measure $\mu(x,\cdot)\equiv \mu_x$ (that is, for every $y\in X$, $\mu_x$ is finite on a neighborhood $V_y$ of $y$),  with positive measure at every element of $\beta_x$. Let further $m:X\times X\to \mathbb{R}_+$ be as in Proposition~\ref{prop-m-slope}, that is, $$m(x,y) = 0 \iff x = y.$$
Finally, let us consider the \textit{local dimension} mapping $n:X\to \mathbb{R}_+$, where we interpret $n(x)$ to be the local dimension of $X$ at $x$. (Obviously, if $X=\mathbb{R}^n$ or if $X$ is a manifold of dimension $n$, then $n(x)\equiv n$, for all $x\in X$.) \smallskip\newline
We are now ready to give the following definition:

\begin{definition}[Dispersion operator]\label{sec04:def:ExtendendedDiffusionOp}
Let $p\in(0,+\infty)$. We define the $p$-dispersion operator $T_{\mu}$ (depending also on $\beta$ and $n:X\to \mathbb{R}_+$) as follows:
\begin{equation}\label{sec04:eq:ExtensionDiffusionOp}
    T_{\mu}[f](x) := \limsup_{B\in \beta_x} \frac{n(x)}{\mu_x(B)} \int_{B}\left|\Delta_f(x,y) \right|^p \mu_x(dy)
\end{equation}
where the limit-superior is taken over the inductive set $\beta_x$ endowed with the partial order of the reverse inclusion and
\begin{equation}\label{eq:Def-Deltaf}
\Delta_f(x,y) := \left\{\begin{array}{cc}
 \frac{f(x) - f(y)}{m(x,y)}, &\text{ if }y\neq x,\\
 \\
 0,&\text{ if }y=x.
\end{array}\right.
\end{equation}
\end{definition}

\begin{remark} $\mathrm{(i).}$ We kept the notation simple and denoted the above dispersion operator by $T_{\mu}$ (rather than $T_{\mu,\beta,m,n,p}$) in order to emphasize that $T_{\mu}$ is the limit-superior of integral operators. The action at $x$ in these operators is integrated by the measure $\mu_{x}$.\smallskip\newline
$\mathrm{(ii).}$  Definition~\ref{sec04:def:ExtendendedDiffusionOp} is inspired by a construction used in~\cite{Sturm1998} to extend diffusion processes to metric spaces. The ``$\limsup$'' ensures that $T_{\mu}$ is always well-defined, with possibly $+\infty$--values. When $X$ is a metric space and $m$ is the distance function, the domain $\dom(T_{\mu})$ contains at least all (locally) Lipschitz functions. This makes the dispersion operator to be a nontrivial extension of \eqref{eq:gamma} beyond the differentiable setting. \smallskip\newline
$\mathrm{(iii).}$  The family $\beta$ in Definition~\ref{sec04:def:ExtendendedDiffusionOp} encompasses several natural choices when the structure of the space $(X,\tau)$ is known. For example, if $(X,\tau)$ is a (pseudo)metric space, then we can take the set of corresponding balls $\beta_x = \{ B(x,r)\}_{r>0}$, for all $x\in X$. More generally, if the topological space $(X,\tau)$ is first-countable, then a natural choice is $\beta_x = \{ \mathcal{V}_n\}_{n\in\mathbb{N}}$, where $\{\mathcal{V}_n\}_{n\in \mathbb{N}}$ is any countable basis of the neighborhoods of~$x$.\smallskip\newline
If $X = \mathbb{R}^n$, then our default choice will be $\beta_x := \{ B(x, r)\}_{r>0}$. 
\end{remark}

We denote by $\mathcal{S}_n^+$ the set of ($n\times n$)--positive semidefinite matrices, and let us consider a map $R:\mathbb{R}^{n}\rightarrow\mathcal{S}_n^{+}$. The following proposition shows that the operators of the form $$\Gamma_R [f](x) = \|R(x)\nabla f(x)\|^2, \,\, x\in \mathbb{R}^n$$ can be obtained as particular cases of~\eqref{sec04:eq:ExtensionDiffusionOp}, under suitable choices of the parameter $p>0$, the separation map $m$, the measure map $\mu:\mathbb{R}^n\times\mathcal{B}(\mathbb{R}^n)\to \mathbb{R}_+$ and a local dimension map $x\mapsto n(x)$. \smallskip\newline
In what follows $\mathrm{supp }\left(\mu_{x}\right)$ stands for the support of the measure $\mu_{x}:=\mu(x,\cdot)$. We say that 
a measure $\mu$ is absolutely continuous with respect to $\nu$ (and denote $\mu<<\nu$) if both measures are defined on the same measurable space $(X,\mathcal{B})$ and it holds: 
$$\nu(A)=0\,\Longrightarrow\,\mu(A)=0,\quad\text{ for all } \,A\in \mathcal{B}.$$
We are now ready to state and prove the following result:

\begin{proposition}\label{sec04:prop:extensionDifussionToMetric}
Let $R:\mathbb{R}^{n}\rightarrow \mathcal{S}_{n}^{+}$ and set $W_{x} :=x+\mathrm{Ker}(R(x))^{\perp}$, for each $x\in \mathbb{R}^n$. Then for $m(x,y):=\Vert x-y\Vert $, and $p=2$, there exist a measure map $\mu :\mathbb{R}^{n}\times \mathcal{B}(\mathbb{R}^{n})\rightarrow \mathbb{R}_{+}$ and a dimension map $n:X\to \mathbb{R}_+$ such that 
$\mathrm{supp}\left( \mu _{x}\right) \subset W_{x}$ for all $x\in \mathbb{R}^{n}$ and 
\begin{equation*}
T_{\mu }[f](x)=\Vert R(x)\nabla f(x)\Vert ^{2},\quad \text{for every }f\in 
\mathcal{C}^{1}(\mathbb{R}^{n}).
\end{equation*}
\end{proposition}

\begin{proof}
Let us fix $x\in \mathbb{R}^{n}$. We are going to define a positive real value $n(x)$ and a measure $\mu _{x}$
whose support is contained in $W_{x}$, in a way that: 
\begin{equation}
T_{\mu }[f](x)=\limsup_{r>0}\frac{n(x)}{\mu _{x}(B(x,r))}\int_{B(x,r)}\Delta_{f}(x,y)^{2}\mu _{x}(dy)=\Vert R(x)\nabla f(x)\Vert ^{2}.  \label{eq:ortega}
\end{equation}
Set $k=\dim \left( \mathrm{Ker\ }R(x)\right) ^{\perp },$ $0\leq k\leq n.$\smallskip\newline
If $k=0,$ then $\mathrm{Ker\ }R(x)=\mathbb{R}^{n},$ $W_{x}\equiv x$ and 
$R(x)\nabla f(x)=0.$ Then (\ref{eq:ortega}) holds trivially by setting $\mu_{x}=\delta _{x}$ (the Dirac measure at $x$) and using the fact that $\Delta_{f}(x,x)=0$ (\textit{cf.} (\ref{eq:Def-Deltaf})).\smallskip \newline
Let us now assume that $1\leq k\leq n.$ Let $\{e_{j}\}_{j=1}^{n}$ be an
orthonormal base of $\mathbb{R}^{n}$ such that
\begin{equation*}
\left( \mathrm{Ker\ }R(x)\right) ^{\perp }=\mathrm{span\ }
(e_{j})_{j=1}^{k}=\mathbb{R}^{k}
\end{equation*}
and
\begin{equation*}
\mathrm{Ker\ }R(x)\equiv\mathbb{R}^{n-k}=\left\{ 
\begin{tabular}{cc}
$\mathrm{span\ }(e_{j})_{j=k+1}^{n},$ & \quad for $k<n$\smallskip 
\\ 
$\{0\}$ & \quad for $k=n$.
\end{tabular}
\right. 
\end{equation*}
Then there exists $R\in \mathcal{S}_{k}^{+}$ (the trace of $R(x)\in \mathcal{S}_{n}^{+}$ on the subspace $\mathbb{R}^{k}\times\{0\}^{n-k}$ of $\mathbb{R}^n$) such that decomposing $z\in \mathbb{R}^{n}$
as $z=(v,w)\in \mathbb{R}^{k}\times \mathbb{R}^{n-k},$ it holds $R(x)z=Rv.$
\smallskip \newline
Let $\Psi :\mathbb{R}^{k}\rightarrow \mathbb{R}^{k}$ be given by
\begin{equation}\label{eq:ort-3}
\Psi (u)=\left\{ 
\begin{array}{cl}
\frac{\Vert u\Vert }{\Vert Ru\Vert }Ru & \text{ if }u\neq 0 \smallskip \\ 
0 & \text{ otherwise}.
\end{array}
\right.  
\end{equation}
Clearly $\Psi $ is an isometric automorphism of $\mathbb{R}^{k}$ (depending
on $x$, which is fixed) with inverse:
\begin{equation*}
\Psi ^{-1}(v)=\left\{ 
\begin{array}{cl}
\frac{\Vert v\Vert }{\Vert R^{-1}v\Vert }R^{-1}v & \text{ if }v\neq 0\smallskip \\ 
0 & \text{ otherwise}.
\end{array}
\right. 
\end{equation*}
In particular $\Psi (B_{k}(0,r))=B_{k}(0,r),$ for every $r>0$. Furthermore, $\Psi $ is a $\mathcal{C}^{1}$-diffeomorphism of $\mathbb{R}^{k}\setminus
\{0\}$.  Following the notation of \cite[Chapter 3]{EG2015Measure}, let us define the Jacobian operator as
\begin{equation}
	J\Psi(u) = |\det (D\Psi(u))|,
\end{equation}
where $D\Psi$ is the derivative of $\Psi$. We define $h_{k}:\mathbb{R}^{k}\rightarrow \mathbb{R}$
(depending on $\Psi$, therefore on $x$) such that
\begin{equation*}
h_{k}(v)=\frac{\Vert R\Psi ^{-1}(v)\Vert ^{2}}{\Vert \Psi ^{-1}(v)\Vert ^{2}}
\,\left[J\Psi (\Psi ^{-1}(v))\right]^{-1},\quad v\in \mathbb{R}^{k}.
\end{equation*}
Notice that for $v=\Psi (u)$ the above yields:
\begin{equation}
h_{k}(\Psi (u))=\frac{\Vert Ru\Vert ^{2}}{\Vert u\Vert ^{2}}\,[J\Psi
(u)] ^{-1},\quad u\in \mathbb{R}^{k}.  \label{eq:ort-1}
\end{equation}
We set
\begin{equation}
\left\{ 
\begin{array}{l}
h:\mathbb{R}^{n}\rightarrow \mathbb{R}\medskip  \\ 
h(z):=\,h_{k}(v),\quad \text{for }z=(v,w)\in \mathbb{R}^{n}
\end{array}
\right.   \label{eq:ort0}
\end{equation}
and consider the measure $\lambda :\mathcal{B}(\mathbb{R}^{n})\rightarrow 
\mathbb{R}_{+}$ (depending on $k=\dim(W_x)$) given by the formula
\begin{equation}
\lambda (A)=\mathcal{L}_{k}\left(A\cap \left(\mathbb{R}^{k}\times \{0\}^{n-k}\right)\right),\qquad \text{for all }A\in 
\mathcal{B}(\mathbb{R}^{n}).  \label{eq:ort-4}
\end{equation}
Let $\pi_k$ denote the projection of $\mathbb{R}^n$ to the first $k$-coordinates. Notice that $\lambda $ is the trivial extension to $\mathcal{B}(\mathbb{R}
^{n})$ of the Lebesgue measure $\mathcal{L}_{k}$ on $\mathcal{B}(\mathbb{R}
^{k})$. \smallskip 
Let us define
\begin{equation}
	\kappa:=\kappa(x) =\mathcal{L}_k(B_k(0,1))^{-1}\left( \int_{B_k(0,1)} \frac{\|Ru\|^2}{\|u\|^2}\mathcal{L}_k(du)\right).
\end{equation}
where $B_{k}(0,r)=\pi _{k}\left( B(0,r)\cap \left(\mathbb{R}^{k}\times \{0\}^{n-k}\right)\right)$. We finally set $n(x) := \kappa(x)\dim(W_x)$ and define the measure $\mu _{x}:\mathcal{B}(\mathbb{R}^{n})~\rightarrow 
~\mathbb{R}_{+}$ as follows:
\begin{equation}
\mu _{x}(A):=\int_{A-x} h(z)\lambda (dz)\equiv \int_{\pi _{k}\left( (A-x)\cap \left(\mathbb{R}^{k}\times \{0\}^{n-k}\right)\right)} h_{k}(v)\,\mathcal{L}_{k}(dv),\qquad \text{for all }A\in \mathcal{B}(\mathbb{R}^{n}).
\label{eq:ort1}
\end{equation}
 
This operation eliminates the last $n-k$ coordinates (which
are equal to $0$ for all elements of $(A-x)\cap  \left(\mathbb{R}^{k}\times \{0\}^{n-k}\right)$), adjusting vectors
to the right dimension for integration. By means of a change of variables induced by $\Psi$ (see, e.g., \cite[Theorem~3.9]{EG2015Measure}), we deduce
\begin{align*}
	\mu_x(B(x,r))& = \int_{B_k(0,r)}h_k(v)\mathcal{L}_{k}(dv)
	 = \int_{B_k(0,r)} h_k(\Psi(u))J\Psi(u)\mathcal{L}_{k}(du)\\
	&= r^{k}\int_{B(0,1)} \frac{\| Ru\|^2}{\|u\|^2} \mathcal{L}_{k}(du)
	= \kappa \mathcal{L}_k(B_k(0,1))r^k = \kappa \mathcal{L}_k(B_k(0,r))
\end{align*}
Now, using the first-order Taylor approximation of $f$ at $x,$ we deduce from \eqref{eq:Def-Deltaf} that
\begin{equation*}
\Delta _{f}(x,y)^{2}=\left[ \left\langle \nabla f(x),\frac{y-x}{\|y-x\|}\right\rangle
+\varepsilon (\|y-x\|)\right] ^{2},\qquad \text{where}\ \lim_{r\rightarrow
0}\,\varepsilon (r)=0.
\end{equation*}
For any $r>0$ we deduce from \eqref{eq:ort1} that:
\begin{align*}
\int_{B(x,r)}\Delta _{f}(x,y)^{2}\mu _{x}(dy)& =\int_{B(0,r)}\left[
\langle \nabla f(x),\frac{z}{\|z\|}\rangle\, +\,\varepsilon (\|z\|)
\right] ^{2}h(z)\lambda (dz) \\
& =\int_{B(0,r)}\langle \nabla f(x),\frac{z}{\|z\|}\rangle
^{2}\,h(z)\,\lambda (dz)\,\,+\\
& \phantom{ds}+ \int_{B(0,r)}2\,\langle \nabla f(x),\frac{z}{\|z\|}\rangle\, \varepsilon (\|z\|)\,h(z)\,\lambda (dz)   +\,\int_{B(0,r)}\varepsilon (\|z\|)^{2}\,h(z)\,\lambda (dz).
\end{align*}
Let $M>0$ be an upper bound of the function $h$ on $B(0,1)$. Since 
$\mu_{x}(B(x,r))= \kappa\mathcal{L}_{k}(B_{k}(0,r))$ and $n(x)=k\cdot\kappa$, it
follows that
\begin{align*}
\frac{2n(x)}{\mu _{x}(B(x,r))}\int_{B(0,r)}\left\langle \nabla f(x),\frac{z}{\|z\|}\right\rangle \varepsilon (\|z\|)h(z)\lambda (dz)& \leq 2\,k\,M\,\Vert
\nabla f(x)\Vert \,\varepsilon (r)\longrightarrow 0\quad \text{and} \\
\text{ }\frac{n(x)}{\mu _{x}(B(x,r))}.\int_{B(0,r)}\varepsilon
(\|z\|)^{2}h(z)\lambda (dz)& \leq \,k\,M\,\varepsilon
(r)^{2}\longrightarrow 0\quad \text{(as }r\rightarrow 0\text{).}
\end{align*}
Denoting by $[\nabla f(x)]_{k}\in \mathbb{R}^{k}$ the vector consisting of
the first $k$-coordinates of $\nabla f(x)$ and recalling the decomposition $z=(v,w)\in \mathbb{R}^{k}\times \mathbb{R}^{n-k}$ we deduce from \eqref{eq:ort0}:
\begin{align*}
\int_{B(0,r)}\left\langle \nabla f(x),\frac{z}{\|z\|}\right\rangle^{2}\, h(z)\lambda (dz)&\,= 
\int_{B(0,r)}\left\langle \nabla f(x),\frac{(v,w)}{\|(v,w)\|}\right\rangle h_{k}(v)\mathcal{L}_{k}(dv).
\end{align*}
Using the change of variables $v=\Psi (u)$ (recall that $\Psi (B(0,r))=B(0,r)
$ for every $r>0$) we obtain from \eqref{eq:ort-3} and \eqref{eq:ort-1}
\begin{align*}
&\int_{B(0,r)}\left\langle \nabla f(x),\frac{z}{\|z\|}\right\rangle
^{2}h(z)\lambda (dz)\\
 =&\int_{B_{k}(0,r)}\left\langle [\nabla
f(x)]_{k},\frac{\Psi (u)}{||\Psi (u)||}\right\rangle ^{2}\frac{\Vert Ru\Vert
^{2}}{\Vert u\Vert ^{2}}\,[J\Psi (u)]^{-1}J\Psi (u)du\, \\
=& \int_{B_{k}(0,r)}\left\langle [\nabla f(x)]_{k},\frac{Ru}{\|u\|}\right\rangle ^{2}du\,=\int_{B_{k}(0,r)}\left\langle R[\nabla
f(x)]_{k},\frac{u}{\|u\|}\right\rangle ^{2}du.
\end{align*}

Therefore we deduce from (\ref{eq:jaime}) and from the definitions of $n(x)$ and $\mu_x$:
\begin{align*}
T_{\mu }[f](x)&=\limsup_{r>0}\frac{n(x)}{\mu_x(B(x,r))}
\int_{B_{k}(0,r)}\left\langle R[\nabla f(x)]_{k},\frac{u}{\|u\|}
\right\rangle ^{2}du\\
&=\limsup_{r>0}\frac{k}{\mathcal{L}_{k}(B_{k}(0,r))}
\int_{B_{k}(0,r)}\left\langle R[\nabla f(x)]_{k},\frac{u}{\|u\|}
\right\rangle ^{2}du\\
&=\|R[\nabla f(x)]_{k}\|^{2}\equiv\Vert R(x)\nabla f(x)\Vert
^{2}.
\end{align*}
The proof is complete.
\end{proof}

\subsection{Oriented dispersion operators \label{sec04-02:OrientedDiffusion}}

The operator $T_{\mu}$ defined in \eqref{sec04:eq:ExtensionDiffusionOp} fails to determine continuous coercive functions, and consequently is not a descent modulus outside the differentiable setting. The reason for this failure will be illustrated in the following example.

\begin{example}\label{sec04:example:DiffusionNotDetermining} Let $X = [-1,1]$ and let $m$ be its usual metric. For each $x\in [-1,1]$, let $\mu(x,\cdot)$ be the usual Lebesgue measure over $[-1,1]$ and $n(x) = 1$. Set $p=2$. 
\[
f(x) = x^2\qquad\text{ and }\qquad g(x) = -x^2.
\]
By \eqref{sec04:eq:IntegroDiff-formula}, we have that 
\[
\forall x\in (-1,1),\quad T_{\mu}[g](x) = T_{\mu}[f](x) = |2x|^2.
\]
Furthermore, it is not hard to see that at $x = \pm 1$, we have that
$$
    T_{\mu}[g](x) = T_{\mu}[f](x) = \lim_{\varepsilon\to 0}\frac{1}{\varepsilon} \int_{1-\varepsilon}^1 \left(\frac{1 - t^2}{1-t}\right)^2dt 
    =\lim_{\varepsilon\to 0}\frac{1}{\varepsilon} \int_{1-\varepsilon}^1 (1+t)^2dt = 4 
$$


Since the only $T$-critical point of $g$ is $0$, we deduce that $T$ does not preserve global minima and so it is not a descent modulus. Furthermore, since the only $T$- critical point of $f$ is $0$ as well, we have constructed two different functions with $T_{\mu}[f] = T_{\mu}[g]$ and that coincide over $\mathcal{Z}_{T}(f)$. In conclusion, $T$ fails to determine continuous coercive functions in general metric spaces, in the sense of Theorem~\ref{sec03:thm:Determination}.
\hfill$\Diamond$
\end{example}

In the above example, the points $x=-1$ and $x=1$ should have been critical for the function $g(x) = -x^2$, since they are global minimizers. However, this fails to be the case because the operator $T_{\mu}$ is not oriented. This leads to the following definition, which induces asymmetry between  descent and  ascent directions (by penalizing the latter). As we shall see, this is particularly relevant in nonsmooth settings.

\begin{definition}[Oriented dispersion operator] \label{def-oriented}Let $\mu$, $\beta$, $m$, $n$ and $p$ be as in Definition~\ref{sec04:def:ExtendendedDiffusionOp}. We define the oriented dispersion operator, denoted by $T^+_{\mu}$, as
\begin{align*}
    T^+_{\mu}[f](x) &:= \limsup_{B\in \beta_x}\frac{n(x)}{\mu_x(B)}\int_{B\cap [f\leq f(x)]} \left[\Delta_f(x,y)\right]^p\mu(x,dy)\\
    &= \limsup_{B\in \beta_x}\frac{n(x)}{\mu_x(B)}\int_{B} \left[\Delta_f^+(x,y)\right]^p\mu(x,dy),
\end{align*}
where 
\begin{equation}\label{eq:Delta+}
\Delta_f^+(x,y) := \left\{\begin{array}{cl}
\frac{ (f(x) - f(y))_+}{m(x,y)}&\text{ if }x\neq y\\
\\
0&\text{ if }x=y.
\end{array}\right.
\end{equation}

\end{definition}

The value $T_{\mu}^+[f](x)$ corresponds to the dispersion of $f$ at $x$ which is exclusively due to the directions of descent. In the smooth case, the value of the oriented dispersion $T_{\mu}^+[f](x)$ is the half of the value of the non-oriented dispersion $T_{\mu}[f](x)$, as expected by symmetry. This is the content of the following proposition.

\begin{proposition} Let $X = \mathbb{R}^n$, $n(x) \equiv n$, $\beta_x = \{ B(x,\varepsilon)\ :\ \varepsilon>0 \}$ and $\mu_x$ be the $n$-dimensional Lebesgue measure for every $x\in\mathbb{R}^n$. Take $p=2$ and $m(x,y) = \|x-y\|$. Then
\[
T^+_{\mu}[f](x) = \frac{1}{2}\|\nabla f(x)\|^2,\quad \text{for every }f\in 
\mathcal{C}^{1}(\mathbb{R}^{n}).
\]
\end{proposition} 

\begin{proof}
If $\|\nabla f(x)\| = 0$, then $0\leq T_{\mu}^+[f](x) \leq T_{\mu}[f](x) = \|\nabla f(x)\|^2 = 0$ and the conclusion follows trivially. \smallskip \newline Let us now consider the case $\|\nabla f(x)\| \neq 0$. By a change of coordinates, we may assume that $x= 0$, $f(0) = 0$ and 
$\nabla f(0) = r\,e_n$, where $r = \|\nabla f(0)\|>0$ and $e_n$ be the $n$-th vector of an orthonormal base of $\mathbb{R}^n$. In this setting, we denote $$S := [f\leq f(0)]\quad \text{and }\quad \mathbb{R}^{n-1}:=\mathrm{span}
\{e_{j}\}_{j=1}^{n-1}\, \equiv \,\{x\in\mathbb{R}^n:\, \langle e_n,z\rangle = 0\}.$$

Following a similar development as in the proof of Proposition~\ref{sec04:prop:extensionDifussionToMetric}, for the particular case $R(x)=\mathbb{I}_n$ (the identity map on $\mathbb{R}^n$), we deduce
\[
T_{\mu }^+[f](x)\,=\,\limsup_{r>0}\frac{n}{\mathcal{L}_{n}(B(0,r))}
\int_{B(0,r)\cap S}\left\langle \nabla f(0),\frac{u}{\|u\|}
\right\rangle ^{2}du.
\]
Consider the semispace $H = \{x\in\mathbb{R}^n:\,\langle e_n, v \rangle\leq 0\}$. Then we have:
\[
\int_{B(0,r)\cap S}\left\langle \nabla f(0), \frac{u}{\|u\|} \right\rangle^2 du = \int_{B(0,r)\cap H}\left\langle \nabla f(0), \frac{u}{\|u\|}\right\rangle^2 du + \int_{B(0,r)\cap (S\triangle H)}\left\langle \nabla f(0), \frac{u}{\|u\|}\right\rangle^2 du\,.
\]
In what follows we show that $\mathcal{L}_n(B(0,r)\cap (S\triangle H))$ is small, where $S\triangle H$ denotes the symmetric difference between $S$ and $H$. To this end, it is easy to see that
\[
B(0,r)\cap (S\triangle H) \subset B_{n-1}(0,r)\times [-d(r),d(r)],
\]
where $d(r)$ stands for the maximal distance between the subspace $\mathbb{R}^{n-1}\times \{0\}$ and the elements of the following set (see Figure~\ref{fig1})
\[
D(r)=\{ (y,z)\in B(0,r):\ y\in B_{n-1}(0,r)\text{ and }f(y,z) = 0\}\,\bigcap\,B(0,r).
\]



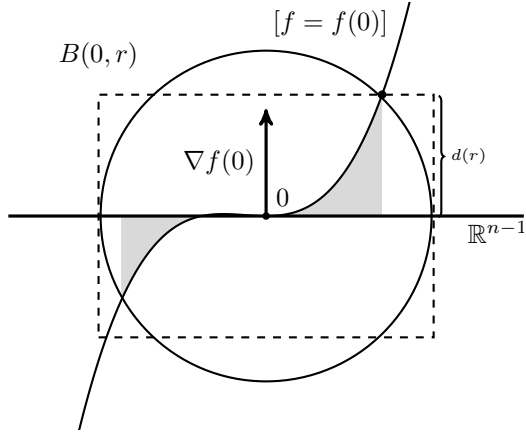
\begin{figure}[h!]\label{fig1}
    \centering
    \begin{tikzpicture}[>=stealth',scale=1]
    \begin{axis}[
	xmin=-2,xmax=2,
	ymin=-2,ymax=2,
	yticklabels=\empty,
	xticklabels = \empty,
	axis line style={draw=none},
	tick style={draw=none}
	]
	\addplot[name path=f,-,thick] expression[domain=-2:2,samples=100,smooth]{x^3+0.5*x^2};

    \path[name path=axis] (axis cs:-2,0) -- (axis cs:2,0);
    
    \addplot [
        thick,
        color=gray,
        fill=gray, 
        fill opacity=0.3
    ]
    fill between[
        of=f and axis,
        soft clip={domain=0:0.9},
    ];
    \addplot [
        thick,
        color=gray,
        fill=gray, 
        fill opacity=0.3
    ]
    fill between[
        of=axis and f,
        soft clip={domain=-1.125:0},
    ];

    \draw[-,very thick] (axis cs:{-2},{0})--(axis cs:{2},{0});
  
    \draw[thick] (axis cs:{0},{0}) circle (2.2cm);	
    
    \node at (axis cs:{-1.3},{1.5}) {\small $B(0,r)$};

	\fill (axis cs:{0},{0}) circle (0.05cm) node[above right]{\small $0$};
    \draw[->,very thick](axis cs:{0},{0})--(axis cs:{0},{1}) node[midway,left]{\small $\nabla f(0)$};
    
    \draw[thick,dashed] (axis cs:{-1.3},{-0.9^3-0.5*0.9^2})--(axis cs:{-1.3},{0.9^3+0.5*0.9^2})--(axis cs:{1.3},{0.9^3+0.5*0.9^2}) -- (axis cs:{1.3},{-0.9^3-0.5*0.9^2})--cycle;

    \node at (axis cs:{0.5},{1.8}) {\small $[f= f(0)]$};
    
    \draw [thick, decorate,decoration={brace,amplitude=2pt,mirror},xshift=21.4pt,yshift=-0.4pt](axis cs:{0.9},{0})--(axis cs:{0.9},{0.9^3+ 0.5*0.9^2}) node[black,midway,xshift=0.4cm] {\tiny $d(r)$};
    \draw[fill] (axis cs:{0.9},{0.9^3+0.5*0.9^2}) circle (0.05cm);
    
    \node at (axis cs:{1.8},{-0.15}) {\small $\mathbb{R}^{n-1}$};
	\end{axis} 
	\end{tikzpicture}
	
	\caption{ The gray area corresponds to the asymmetric difference $S\triangle H$. The dashed line outlines the set $B_{n-1}(0,r)\times [-d(r),d(r)]$, where the dot depicts the farthest point of the set $D(r)$ to the linear subspace $\mathbb{R}^{n-1}$.}
\end{figure}

Using the Implicit Function Theorem, we deduce the existence of an open subset $\mathcal{U}\subset \mathbb{R}^{n-1}$ containing~$0$, an open set $\mathcal{V}\subset \mathbb{R}^n$ containing $0$ and a function $\varphi:\mathcal{U}\to \mathbb{R}$ of class $\mathcal{C}^1$ such that its graph coincides with $[f = 0]\cap \mathcal{V}$ and 
\[
\nabla \varphi(0) = - (\partial_n f(0))^{-1} \begin{pmatrix}
\partial_1 f(0)\\
\vdots\\
\partial_{n-1}f(0)
\end{pmatrix} = \mathbf{0}_{n-1}.
\]

Therefore, for $r>0$ sufficiently small, we have $B_{n-1}(0,r)\subset \mathcal{U}$ and $B(0,r)\subset \mathcal{V}$, which yields
\[
\{ (y,z)\in B_{n-1}(0,r)\times \mathbb{R} : \, f(y,z) = 0\} = \{ (y,\varphi(y)):\, y\in B(0,r) \}.
\]
Therefore $d(r) = \sup \{ |\varphi(y)|:\, y\in B_{n-1}(0,r) \}$.  Evoking the mean value theorem we deduce
\[
\sup\{ |\varphi(y)|: \, y\in B_{n-1}(0,r) \} \,\leq \, r\cdot \sup\{ \|\nabla\varphi(y)\|:\, y\in B_{n-1}(0,r) \}.
\]
By continuity of $\nabla \varphi$ and recalling that $\nabla\varphi(0) = 0$ we deduce that $d(r) = o(r)$. Recalling formula~\eqref{eq:VolumeNBall} for the volume of the ($n-1$)-dimensional ball $B_{n-1}(0,r)$, we set
$$ K = \frac{2\pi^{(n-1)/2}}{\Gamma\left( \frac{n-1}{2} + 1 \right)} $$ 
and we obtain:
\[
\mathcal{L}_n (B(0,r)\cap (S\triangle H)) \leq \mathcal{L}_{n-1}(B_{n-1}(0,r))\cdot 2d(r) = (K\cdot r^{n-1})\,o(r) \equiv o(r^{n})
\]
Therefore,
\begin{align*}
    \frac{n}{\mathcal{L}_n(B(0,r))} \int_{B(0,r)\cap (S\triangle H)}\left\langle \nabla f(0), \frac{u}{\|u\|}\right\rangle^2 du\,    
    &\leq n\,\|\nabla f(0)\|^2 \,\frac{\mathcal{L}_n (B(0,r)\cap (S\Delta H))}{\mathcal{L}_n (B(0,r))} \smallskip\\
    &= \frac{n\|\nabla f(0)\|^2}{K_n}\frac{o(r^n)}{r^n} \xrightarrow{r\to 0} 0.
\end{align*}

Since $B(0,r)\cap H$ is the south--half of the ball $B(0,r)$, a symmetry argument ensures 
\begin{align*}
 T^+_{\mu}[f](0) &= \limsup_{r\to 0} \frac{n}{\mathcal{L}_n(B(0,r))} \, \int_{B(0,r)\cap H}\left\langle \nabla f(0), \frac{u}{\|u\|}\right\rangle^2 du \smallskip \\
 &= \frac{1}{2}\,\limsup_{r\to 0}\frac{n}{\mathcal{L}_n(B(0,r))}\,\int_{B(0,r)}\left\langle \nabla f(0), \frac{u}{\|u\|}\right\rangle^2 du\,  = \frac{1}{2} \|\nabla f(0)\|^2.
\end{align*}
The proof is complete.
\end{proof}

\begin{remark} The above arguments can be easily adapted to show that when $\mu_x$ and $n(x)$ are as in Proposition~\ref{sec04:prop:extensionDifussionToMetric}, then the oriented dispersion operator $T_{\mu}^+$ (\textit{cf.} Definition~\ref{def-oriented}) satisfies:
\[
T_{\mu}^+[f](x) = \frac{1}{2}\|R(x)\nabla f(x)\|^2, \qquad\text{for all } f \in \mathcal{C}^1.
\]
\end{remark}

The following proposition justifies the introduction of the oriented dispersion in a nonsmooth setting. Given a metric space $(X,d)$ we denote by $\mathrm{Lip}(X)$ the class of real-valued Lipschitz continuous functions on $X$. 

\begin{theorem}\label{sec04:thm:DispersionIsModDescent} Let $(X,d)$ be a metric space and $\mu:X\times\mathcal{B}(X)\to\mathbb{R}_+$ a measure mapping such that for every $x\in X$, $\mu_x$ is a locally finite measure with positive measure on open sets and $\beta=\{\beta_x\}_{x\in X}$ be any family of neighborhood bases. Then, the oriented dispersion operator $T^+_{\mu}$ is a descent modulus for $\mathcal{K}(X)$ and verifies that $\mathcal{K}(X)\cap\mathrm{Lip}(X) \subset \dom(T^+_{\mu})$. 
\end{theorem}

\begin{proof}
The conditions over $\mu$ ensure that $\mathcal{K}(X)\cap\mathrm{Lip}(X)\subset \dom(T^+_{\mu})$. Clearly the operator $T^+_{\mu}$ preserves global minima and is scalar-monotone. Let us now show that it is monotone.
Let $f,g\in\mathcal{K}(X)$ and let $x\in X$ such that
\[
(f(x) - f(z))_+ \geq (g(x) - g(z))_+.
\]
Then, $\Delta_f^+(x,z) \geq \Delta_g^+(x,z)$ for all $z\in X$ and the conclusion follows.
\end{proof}

\begin{remark} The measure map $\mu:X\times\mathcal{B}(X)\to\mathbb{R}_+$ is assumed to be locally finite, which yields in particular that each measure $\mu_x$ is finite on the compact sets of $(X,\tau)$. 
Apart from this assumption and the existence of a neighborhood system $\{\beta_x \}_{x\in X}$ where $\mu_x$ takes nonzero values, no other property is required. In this setting, the superior limit in Definition~\ref{sec04:def:ExtendendedDiffusionOp} and Definition~\ref{def-oriented} are well-defined, yielding that the dispersion operators are descent moduli. Even less will be required to define nonlocal operators (see next section), namely, $\mu_x$ to be finite on compact sets.
\end{remark}

\subsection{Oriented nonlocal operators}

Apart from the diffusion operators, which are of local nature, one can also consider nonlocal operators. These latter serve to model jump dynamics, see \textit{e.g.} \cite{EK1986}. We shall now define dispersion measures for these processes.

\begin{definition}[Nonlocal dispersion operators]\label{sec04:def:nonlocalDispersion} Let $\mu:X\times\mathcal{B}(X)\to\mathbb{R}_+$ be a  measure mapping such that for every $x\in X$, $\mu_x$ is finite on all compact sets, and let $\phi:\mathbb{R}_+\to\mathbb{R}_+$ be a strictly increasing function with $\phi(0) = 0$. We define the nonlocal dispersion operator induced by $\phi$ and $\mu$ as
\[
T_{\phi,\mu}[f](x) = \int_X \phi(|f(x) - f(y)|)\mu(x,dy), 
\]
\end{definition}

By construction, $T_{\phi,\mu}$ is finite for every measurable bounded function with compact support. When $X=V$ is a finite space, the nonlocal operators are particularly relevant, due to the fact that all points are isolated and so diffusion is not possible. In this setting, the measure map $\mu$ can be represented by a matrix $L:V\times V\to\mathbb{R}_+$, in the form of

\[
T_{\phi,\mu}[f](x) = \sum_{y\in V} L(x,y) \,\phi(|f(x) - f(y)|)
\]
\begin{remark}
In the context of Markov generators, the nonlocal operators are of the form $$L[f](x) = \int_X (f(y) - f(x))\mu_x(dy),$$ where $\mu$ is assumed to be regular in the sense that $x\mapsto \mu_x(A)$ is measurable for every $A\in\mathcal{B}(X)$. When 
$\phi(t) = t^2$ we are working with the dispersion operator $T_{\phi,\mu}[f] := \Gamma[f]$, where $\Gamma$ is the carr\'e-du-champ operator associated to $L$, which in all generally is defined by the identity $\Gamma[f] = L[f^2] - 2fL[f]$ (as soon as $f,f^2\in \dom (L)$, see e.g. \cite{BGL2014}).
\end{remark} 

In general, a nonlocal operator $T_{\phi,\mu}$ might fail to be a descent modulus and to determine functions in the sense of Theorem~\ref{sec03:thm:Determination}.  

\begin{example}\label{example:CarreFailsDetermination} Let $\phi(t):= t^2$. Fix $N\in\mathbb{N}$ even, and set $\mathcal{V}:= \mathbb{Z}_{N}\cup\{\bar{0}\}$, where $\bar{0}\notin \mathbb{Z}_N$ and $\mathbb{Z}_N =\{0,1,\ldots, N-1\}$ stands for the usual cyclic additive group modulo $N$. We define a nonlocal operator $L$ as follows:
\[
L(x,y) = \begin{cases}
\phantom{jo}1/2,\qquad & \text{ if }x\in \mathbb{Z}_N\setminus\{0\}\text{ and }y = x \pm 1,\\
\phantom{jo}1/3,\qquad &\text{ if }x = 0\text{ and }y\in\{\bar{0},1,-1\},\\
\phantom{jo}1,\qquad &\text{ if }x=\bar{0} \text{ and }y = 0,\\
\phantom{jo}0,&\text{ otherwise.}
\end{cases}
\]
\begin{figure}[h]
	\fontsize{6pt}{6pt}\selectfont
	\centering
	\begin{tikzpicture}[->,>=stealth',shorten >=0pt,auto,node distance=2cm,
		semithick]
		
		\node[state] 		 (A){ $\bar{0}$};
		\node[state] 		 (B) [right of=A] {$0$};
		\node[state] 		 (C) [above right of=B] {$1$};
		\node[state] 		 (D) [right of =C] {$2$};
		\node[state]         (E) [below right of=D] {$3$};	
		\node[state]         (F) [below left of=E] 	{$4$};
		\node[state]         (G) [left of=F] {$5$};

		\path (A) [bend right=20] edge node[below]{$1$} (B) 
		(B) edge node[above]{$\tfrac{1}{3}$} (A)
		(B) edge node[below]{$\tfrac{1}{3}$} (C)
		(C) edge node[above]{$\tfrac{1}{2}$} (B)
		(C) edge node[below]{$\tfrac{1}{2}$} (D)
		(D) edge node[above]{$\tfrac{1}{2}$} (C)
		(D) edge node[below]{$\tfrac{1}{2}$} (E)
		(E) edge node[above]{$\tfrac{1}{2}$} (D)
		(E) edge node[above]{$\tfrac{1}{2}$} (F)
		(F) edge node[below]{$\tfrac{1}{2}$} (E)
		(F) edge node[above]{$\tfrac{1}{2}$} (G)
		(G) edge node[below]{$\tfrac{1}{2}$} (F)
		(G) edge node[above]{$\tfrac{1}{2}$} (B)
		(B) edge node[below]{$\tfrac{1}{3}$} (G)
		;

	\end{tikzpicture}
	\caption{Case $N=6$.}
	\label{fig:ExampleZN}
\end{figure}
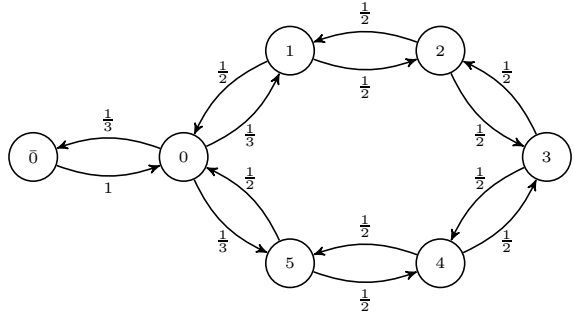
Now, choose two functions $f_1,f_2\in\mathbb{R}^\mathcal{V}$ satisfying that
\[
f_i(\bar{0}) = f_i(0) = 0\quad\text{ and }\quad |f(x\pm 1) - f(x)| = 1,\forall x\in \mathbb{Z}_N.
\]
There is at least ${ N\choose N/2} > 1$ functions verifying the above requirements, so we can take $f_1\neq f_2$. However, it is not hard to see that for the measure map $\mu$ associated with $L$, the nonlocal operator $T_{\mu}$ verifies that
\[
T_{\mu}[f_i] (x) = \begin{cases}
0\quad&\text{ if }x=\bar{0},\\
2/3\quad&\text{ if }x= 0,\\
1&\text{ otherwise.}
\end{cases}
\]
for $i=1,2$. Thus, $T_{\mu}$ does not preserve the global minima since either $\argmin f_i \supseteq \{ 0,\bar{0}\}$ or $\argmin f_i \subset \mathcal{V}\setminus \{ 0,\bar{0}\}$. Furthermore, $T_{\mu}$ fails the determination theorem even for functions with $\mathcal{Z}_{T_{\mu}}(f)\neq \emptyset$.
\hfill$\Diamond$
\end{example}
\smallskip

Example~\ref{example:CarreFailsDetermination} is very illustrative as concerns the following: when $\phi(t) = t^2$, nonlocal operators do not preserve global minima in general. Indeed, if $\mathcal{V}$ is a finite state space, $T_{\mu}[f](x)$ measures the dispersion around point $x\in \mathcal{V}$, when $L(x,y)$ represents the probability to jump from the point~$x$ to the point~$y$. Therefore it is natural for $T_{\mu}[f](x)$ to be strictly positive. However, by imposing $f(0)=f(\bar 0)$ in Example~\ref{example:CarreFailsDetermination} we are forcing a point with no dispersion: starting from $x=\bar{0}$, the only possibility is to jump to~$0$.

\begin{definition}[Oriented nonlocal operators]\label{sec04:def:OrientedNonlocal}
Let $\phi$ and $\mu$ be as in Definition~\ref{sec04:def:nonlocalDispersion}. We define the oriented nonlocal operator induced by $\phi$ and  $\mu$ as
\begin{equation}
    T_{\phi,\mu}^+[f](x) = \int_{[f\leq f(x)]} \phi(f(x) - f(y))\mu_x(dy) = \int_{X} \phi((f(x) - f(y))_+)\mu_x(dy)
\end{equation}
\end{definition}

Similarly to the (local) oriented dispersion operator, the above operator is a descent modulus for $\mathcal{K}(X)$, and always determines a suitable subclass of continuous coercive functions $\mathcal{K}(X)$. In the nonlocal case, we do not need to assume Lipschitz continuity, and consequently, $X$ can be a mere topological space. The class is given by the \textit{strictly coercive} functions, which is given by
\begin{equation}\label{eq:StrictlyCoercive}
\mathcal{K}_s(X) = \left\{ f:X\to \mathbb{R}\ :\ \forall x\in X,\, [f\leq f(x)]\text{ is compact} \right\}.
\end{equation}
The main difference between $\mathcal{K}_s(X)$  and $\mathcal{K}(X)$ is that the latter class admits functions attaining their maximum value since the set $[f \leq \max_X f]$ does not have to be compact. If $X$ is compact, then the classes $\mathcal{K}(X)$  and $\mathcal{K}_s(X)$ coincide, however, if $X$ is noncompact, then functions in $\mathcal{K}_s(X)$ cannot attain their supremum.

\begin{theorem}\label{thm:CarreNonlocalDetermining} Let $\mu:X\times\mathcal{B}(X)\to\mathbb{R}_+$ be a measure mapping such that $\mu_x$ is finite on all compact sets, for every $x\in X$. Then the oriented nonlocal operator $T_{\phi,\mu}^+$ is a descent modulus for $\mathcal{K}(X)$ and verifies that $\mathcal{K}_s(X)\subset \dom(T_{\phi,\mu}^+)$.
\end{theorem}

\begin{proof} 
Since $\mu_x$ is finite on all compact sets, for each $x\in X$, we deduce that $\mathcal{K}_s(X)\subset \dom(T_{\phi,\mu}^+)$. Furthermore, since $\phi(0) = 0$, it is clear that $T_{\phi,\mu}^+$ preserves global minima. 

Let us show now that $T_{\phi,\mu}^+$ is monotone: let $f,g\in \mathcal{K}(X)$ and $x\in X$ such that
\[
(f(x)-f(z))_+\geq (g(x) - g(z))_+,\quad \forall z\in X. 
\]
Then, since $\phi$ is non-decreasing, we have that
\[
T_{\phi,\mu}^+[f](x) = \int_X\phi( (f(x)-f(z))_+ )\mu_x(dz) \geq \int_X\phi( (g(x)-g(z))_+ )\mu_x(dz) = T_{\phi,\mu}^+[g](x).
\]
We conclude that $T_{\phi,\mu}^+$ is monotone.

Finally, let us show that $T_{\phi,\mu}^+$ is scalar-monotone. Let $f\in \mathcal{K}(X)$, let $r>1$ and let $x\in X$ such that $0<T[f](x)<+\infty$. By monotonicity, we have that $T_{\phi,\mu}^+[rf](x) \geq T_{\phi,\mu}^+[f](x)>0$. Let us now define the sets
\[
A_{n} = \left\{ z\in X: \, f(x)- f(z) \geq \frac{1}{n}  \right \}\, \bigcap\, \left \{ z\in X\ :\ \phi(r(f(x)- f(z))) - \phi(f(x)- f(z)) \geq  \frac{1}{n}  \right \}.
\]
Clearly $\{A_n\}_n$ is an increasing sequence of $\mu_x$-measurable sets satisfying: $$\bigcup_{n\geq 1} A_n = [f<f(x)].$$ Thus, by monotone convergence theorem, we have that
\[
\lim_{n}  \int_{A_n} \phi(f(x) - f(z))\mu_x(dz) = \int_{[f<f(x)]}\phi(f(x) - f(z))\mu_x(dz) = T_{\phi,\mu}^+[f](x)>0.
\]
Choose then $n\in\mathbb{N}$ such that  $\int_{A_n} \phi(f(x) - f(z))\mu_x(dz)> 0$. Then, $\mu(x,A_n)>0$ and
\begin{align*}
 T_{\phi,\mu}^+[rf](x) &=  \int_X  \phi(r(f(x) - f(z))_+)\mu_x(dz)\\
 &= \int_{A_{n}}  \phi(r(f(x) - f(z)))\mu_x(dz) + \int_{X\setminus A_n}  \phi(r(f(x) - f(z))_+)\mu_x(dz)\\
 &\geq \int_{A_{n}}  \phi((f(x) - f(z))) + \frac{1}{n}\mu_x(dz) + \int_{X\setminus A_n}  \phi((f(x) - f(z))_+)\mu(x,dz)\\
 &= \int_X  \phi((f(x) - f(z))_+)\mu(x,dz) +\frac{1}{n}\mu(A_n)
 > T_{\phi,\mu}^+[f](x).
\end{align*}

All three properties of Definition~\ref{sec03:def:ModulusOfDescent} are satisfied and the proof is complete.
\end{proof}

\begin{remark}
Any $\Gamma$-operator (carr\'e-du-champ operator) coming from a regular Markov generator in $\mathbb{R}^n$ (with the euclidean distance) has the form
\[
\Gamma[f](x) = \limsup_{r\to 0} \frac{n(x)}{\mu_{1,x}(B(x,r))} \int_{B(x,r)}\left[ \Delta_f(x,y) \right]^2 \mu_{1,x}(dy) + \int_X (f(x) - f(y))^2\mu_{2,x}(dy).
\]
The above operator measures the dispersion of the function $f$ around a point $x$, when the point evolves following a local diffusion process linked to $(\mu_{1,x})_x$ and a nonlocal jump process given by $(\mu_{2,x})_x$. The oriented dispersion is only taking into account the descent directions and has the form
\[
\Gamma^+[f](x) = \limsup_{\varepsilon\to 0} \frac{n(x)}{\mu_{1,x}(B(x,\varepsilon))} \int_{B(x,\varepsilon)}\left[ \Delta_f^+(x,y) \right]^2 \mu_{1,x}(dy) + \int_X [(f(x) - f(y))_+]^2\mu_{2,x}(dy)
\]
The above oriented operator is a descent modulus for $\mathcal{K}(X)$. Thus, if we know the oriented dispersion of a continuous coercive function $f$ (with finite oriented dispersion), and we know its values on the critical points (that is, points with zero oriented dispersion), we completely determine the function $f$, in the spirit of Theorem~\ref{sec03:thm:Determination}. 
\end{remark}


\section{Descent moduli over finite sets}\label{sec:5}

Finite state spaces provide a simple and experimental framework to
investigate further properties of moduli of descent. We shall use the
terminology \textquotedblleft \textit{finite descent modulus}" to refer to a
descent modulus over a finite set. In this section we study two particular
features of finite descent moduli:

\begin{itemize}
\item an alternative proof, based on a probabilistic approach, of (an
enhanced version of) the determination theorem for descent moduli mimicking
Markov generators; and

\item a characterization of \textit{homogeneous} finite descent moduli, up
to a natural equivalence relation based on the corresponding critical map.
\end{itemize}

\medskip We have already encountered a finite descent moduli in Example~\ref{example:CarreFailsDetermination}. Let us present a general procedure generating finite descent moduli: on a finite state space $\mathcal{V}$ (neither empty nor a
singleton), consider a Markov generator $L:=(L(x,y)_{x,y\in \mathcal{V}}$,
namely a matrix satisfying 
\begin{equation*}
\left\{ 
\begin{array}{lcl}
\forall x,y\in \mathcal{V}:\,\,\quad x\neq y\,\,\Longrightarrow \,\,L(x,y) & \geq & 0 \\
[2mm] 
\forall x\in \mathcal{V}:\qquad \phantom{salas}\sum_{y\in \mathcal{V}}L(x,y) & = & 0
\end{array}
\right.
\end{equation*}

Such a generator acts linearly on any function $f\in \mathbb{R}^{\mathcal{V}}$ (which coincides with $\mathcal{K}(\mathcal{V})$)  via 
\begin{equation}
\forall \ x\in \mathcal{V},\qquad L[f](x)=\sum_{y\in \mathcal{V}}L(x,y)(f(y)-f(x))  \label{eq:star}
\end{equation}

By analogy to Definition~\ref{def-oriented} and Definition~\ref{sec04:def:OrientedNonlocal}, we consider the non-linear operator $T_{L}$ acting on any function $f\in \mathbb{R}^{\mathcal{V}}$ via 
\begin{equation}\label{TL}
\forall \ x\in \mathcal{V},\qquad T_{L}[f](x)=\sum_{y\in \mathcal{V}}L(x,y)(f(x)-f(y))_{+}
\end{equation}

From Theorem \ref{thm:CarreNonlocalDetermining}, $T_{L}$ is a descent modulus. In Subsection~\ref{apa}, we will recover the determination theorem for this kind of
descent modulus via a probabilistic approach.\smallskip

More generally, for any $m>0$, one can consider $T_{L,m}$ given by 
\begin{equation}
\forall \ x\in \mathcal{V},\qquad T_{L,m}[f](x)=\left( \sum_{y\in \mathcal{V}}L(x,y)((f(x)-f(y))_{+})^{m}\right) ^{1/m}  \label{eq:TLp}
\end{equation}
as well as its limit $T_{L,\infty }$ as $m$ goes to infinity: 
\begin{equation}
\forall \ x\in \mathcal{V},\qquad T_{L,\infty }[f](x)=\max
\{(f(x)-f(y))_{+}\,:\,y\in D_{x}\}  \label{eq:TLi}
\end{equation}
where for every $x\in \mathcal{V}$ we set:
\begin{equation}\label{eq:act}
D_{x}:=\{x\}\sqcup \{y\in \mathcal{V}\,:\,L(x,y)>0\}
\end{equation}

Let us recall (see Definition~\ref{def-homog} for $p=1$) that a descent
modulus $T$ is said to be \textit{homogeneous} (or 1--\textit{homogeneous})
if for all $r\geq 0$ and $f\in \mathbb{R}^{\mathcal{V}}$ we have: $T[rf]=rT[f]$.\smallskip\newline
All the above operators $T_{L,m}$, $m\in (0,+\infty ],$ are homogeneous
descent moduli. It should be noticed that there are many more homogeneous
descent moduli: for instance in \eqref{eq:TLp} we can allow the exponent $m$
to depend on $x\in \mathcal{V}$. Moreover, given $n$ homogeneous descent
moduli $T_{1}$, ..., $T_{n}$, and positive numbers $a_{1},...,a_{n}>0$, the
weighted sum $a_{1}T_{1}+\cdots +a_{n}T_{n}$ is again a homogeneous descent
moduli. Even fancier constructions are possible. This being said, there
exist non-homogeneous descent moduli. Indeed, for any non-decreasing mapping $\phi \,:\,\mathbb{R}_{+}\rightarrow \mathbb{R}_{+}$ with $\phi (0)=0$, the
descent modulus $T_{\phi }$ defined by 
\begin{equation*}
\forall \ x\in \mathcal{V},\qquad T_{L}[f](x)=\sum_{y\in \mathcal{V}
}L(x,y)\phi ((f(x)-f(y))_{+})
\end{equation*}
is homogeneous if and only if $\phi $ is linear, as long as $L\neq 0$
.\smallskip

Given a descent modulus $T$ we recall from (\ref{eq:ZTf}) the \textit{critical map} $\mathcal{Z}_{T}$, which associates to every function $f\in \mathbb{R}^{\mathcal{V}}$ its set of critical points $\mathcal{Z}_{T}(f)=(T[f])^{-1}(0)$. Notice that the critical maps $\mathcal{Z}_{T_{L,m}}$ related to the
moduli $T_{L,m}$ in \eqref{eq:TLp}--\eqref{eq:TLi} are all the same as $m$
varies in $(0,+\infty ]$.\smallskip\ 

In Subsection~\ref{cm} we introduce an equivalence relation among
homogeneous descent moduli, using the critical maps. Under this relation,
all moduli $T_{L,m}$ in \eqref{eq:TLp} turn out to be equivalent to each
other (for different values of $m\in \mathbb{N}$) and also equivalent to $T_{L,\infty }.$ The main result of this section is to show that every
homogeneous descent modulus on a general finite set $\mathcal{V}$ (without
generator $L$) is still of the form \eqref{eq:TLi} for some family $\mathcal{D}=\{\mathcal{D}_{x}\}$ which is naturally associated to~$T$, provided it
satisfies a (necessary and sufficient) mild condition.

\subsection{A probabilistic approach}

\label{apa}

Let $L:=(L(x,y))_{x,y\in \mathcal{V}}$ be a Markov generator on the finite
set $\mathcal{V}$.

\smallskip For every $f\in \mathbb{R}^{\mathcal{V}}$ the associated $f$-oriented Markov generator $L^{f}:=(L^{f}(x,y))_{x,y\in \mathcal{V}}$ is
defined for $x,y\in \mathcal{V}$ with $x\neq y$ as follows:
\begin{equation*}
L^{f}(x,y):=\left\{ 
\begin{array}{ll}
L(x,y), & \hbox{if $f(y)\leq f(x)$} \\[2mm]
\phantom{dav}0, & \hbox{otherwise.}
\end{array}
\right. 
\end{equation*}
The values $L^{f}(x,x)$ on the diagonal are determined by the fact
that the sum of the rows $\sum_{x\in\mathcal{V}}L(x,y)$ should  vanish.\smallskip 

Let $T:\mathbb{R}^{\mathcal{V}}\rightarrow \mathbb{R}^{\mathcal{V}}$ be
defined for every $f\in \mathbb{R}^{\mathcal{V}}\ $and $x\in \mathcal{V}$ as
follows:
\begin{equation*}
T[f](x):=-L^{f}[f](x)=-\sum_{y\in \mathcal{V}}L^{f}(x,y)(f(y)-f(x))=\sum_{y\in \mathcal{V}}L(x,y)(f(x)-f(y))_{+}.
\end{equation*}
This non-linear operator $T$ coincides with $T_L$ defined in \eqref{TL} and is a descent modulus. For every $f\in \mathbb{R}^{\mathcal{V}}$ the set of $T$-critical points is given by
the formula
\begin{equation}
\mathcal{Z}_{T}(f):=\{x\in \mathcal{V}\,:\,T[f](x)=0\}.  \label{eq:ZT}
\end{equation}

Given $x,y\in \mathcal{V},$ an $L$-path from $x$ to $y$ is a finite sequence 
$\{x_{k}\}_{0\leq k\leq N}$ with $N\geq 0$, $x_{0}=x$, $x_{N}=y$ and such
that for all $0\leq k<N$, $L(x_{k},x_{k+1})>0$. This path is called an $L^{f}
$-path from $x$ to $y$ if in addition $\{f(x_{k})\}_{0\leq k\leq N}$ is a
non-increasing finite sequence. We write $x\overset{f}{\rightarrow }y$ to
indicate that there exists a $L^{f}$-path from $x$ to $y.$ We set: 
\begin{equation*}
x\succeq _{f}y\,\Longleftrightarrow \,x\overset{f}{\rightarrow }y\qquad 
\text{and}\qquad x\approx _{f}y\,\Longleftrightarrow \,\left\{ 
\begin{array}{c}
x\overset{f}{\rightarrow }y \\ 
y\overset{f}{\rightarrow }x
\end{array}
\right. 
\end{equation*}
It is straighforward to check that $\succeq _{f}$ is an order relation on $\mathcal{V}$ and $\approx _{f}$ is its corresponding equivalence relation ($x\approx _{f}y$ if and only if $x\succeq _{f}y$ and $y\succeq _{f}x$). The
set of minima of $\succeq _{f}$ is defined as follows:
\begin{equation}\label{eq:M}
M(f):=\left\{ \bar{x}\in \mathcal{V}:\;\,\forall x\in \mathcal{V},\;\left( 
\bar{x}\succeq _{f}x\,\Rightarrow \bar{x}\approx _{f}x\right) \right\} .
\end{equation}
Notice that $\bar{x}\in M(f)$ if and only if for any $y\in \mathcal{V}$ with 
$f(y)<f(\bar{x})$ and any $L$-path $\{x_{k}\}_{0\leq k\leq N}$ from $x$ to $y
$, we have $\max_{0\leq k\leq N}f(x_{k})>f(\bar{x})$. Moreover, we always
have $M(f)\subset \mathcal{Z}_{T}(f)$ and the inclusion may be strict.

\begin{example}
	Let $\mathcal{V} = \mathbb{Z}_9$, and set $L$ such that 
	\[
	L(x,y) >0 \iff y=x\pm 1.
	\] 
	Consider $f = (1,0,0,1,2,1,1,2,1)$. The set $\mathcal{V}$, its connections through $L$ and the level sets of $f$ are depicted in Figure \ref{fig:ExampleLF}.
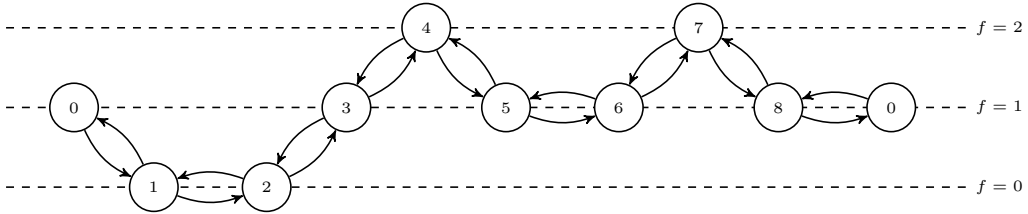
\begin{figure}[h]
	\fontsize{6pt}{6pt}\selectfont
	\centering
	\begin{tikzpicture}[->,>=stealth',shorten >=0pt,auto,node distance=1.5cm,
		semithick]
		
		
		\node (L1) {};
		\node (L2)  [node distance = 1.06 cm, above of = L1]{};
		\node (L0)  [node distance = 1.06 cm, below of = L1]{};
		
		\node (L1end) [node distance = 13.3cm,right of = L1]{$f = 1$};
		\node (L2end)  [node distance = 13.3cm,right of = L2]{$f = 2$};
		\node (L0end)  [node distance = 13.3cm,right of = L0]{$f = 0$};
		
		\path (L1) edge [-,dashed] (L1end)
		(L0) edge [-,dashed] (L0end)
		(L2) edge [-,dashed] (L2end)
		;
		 
		\node[state] 		 (A)[fill = white, right of = L1,node distance = 1cm]{ $0$};
		\node[state] 		 (B) [fill = white,below right of=A] {$1$};
		\node[state] 		 (C) [fill = white,right of=B] {$2$};
		\node[state] 		 (D) [fill = white,above right of =C] {$3$};
		\node[state]         (E) [fill = white,above right of=D] {$4$};	
		\node[state]         (F1) [fill = white,below right of=E] 	{$5$};
		\node[state]         (F2) [fill = white,right of=F1] 	{$6$};
		\node[state]         (G) [fill = white,above right of=F2] {$7$};
		\node[state]         (H) [fill = white,below right of=G] {$8$};
		\node[state]         (I) [fill = white, right of=H] {$0$};
		
		\path (A) [bend right=20] edge node{} (B) 
		(B) edge node{} (A)
		(B) edge node {}(C)
		(C) edge node{} (B)
		(C) edge node{} (D)
		(D) edge node{} (C)
		(D) edge node{} (E)
		(E) edge node{} (D)
		(E) edge node{} (F1)
		(F1) edge node{} (E)
		(F1) edge node{} (F2)
		(F2) edge node{} (F1)
		(F2) edge node{} (G)
		(G) edge node{} (F2)
		(G) edge node{} (H)
		(H) edge node{} (G)
		(H) edge node{} (I)
		(I) edge node{} (H)
		;
	\end{tikzpicture}
	\caption{Node $0$ has been replicated at the beginning and end of the representation. Only connections $(x,y)$ with $L(x,y)>0$ have been drawn. }
	\label{fig:ExampleLF}
\end{figure}

Here, $M(f) = \left\{1,2,5,6\right\}$ and $\mathcal{Z}_T(f) = \{ 1,2,5,6,8 \}$. The node $8$ is critical since $L$ does not allow to jump to any node with smaller value in one step. However, the path $8\to0\to1$ is an $L^f$-path leading to a point with smaller $f$-value. Note that $5$ and $6$ are in $M(f)$ since there is no $L^f$-path emanating from any of them and landing at a different node with smaller $f$-value.\hfill$\diamond$
\end{example}

\smallskip 

For $x\in \mathcal{V}$, let $X_{x}^{f}:=(X_{x}^{f}(t))_{t\geq 0}$ stand for
a Markov process starting from $x$ and whose generator is $L^{f}$. For such
a process, the function 
\begin{equation*}
\mathbb{R}_{+}\ni t\mapsto f(X_{x}^{f}(t))
\end{equation*}
is almost surely
non-increasing and bounded, thus converging.
Furthermore, the finite Markov process $X_{x}^{f}(t)$ is converging in law
for large $t\geq 0$ toward a distribution which may depend on the initial
point $x$ and whose support is included into the set $M(f).$\medskip 

Fix $f,g\in \mathbb{R}^{\mathcal{V}}$. Since $\mathcal{V}$ is finite, the
functions $f,g$ are trivially continuous and coercive. Therefore, Theorem~\ref{sec03:thm:Determination} directly yields:
\begin{equation}
\left. 
\begin{array}{c}
T[f]=T[g] \smallskip \\ 
f=g\text{ \ on }\mathcal{Z}_{T}(f)
\end{array}
\right\} \Longrightarrow f=g.  \label{eq:dm}
\end{equation}

In what follows, we obtain~\eqref{eq:dm} via a probabilistic approach, in a slightly enhanced version, namely replacing the
set $\mathcal{Z}_{T}(f)=\mathcal{Z}_{T}(g)$ (where $f$ and $g$ are assumed to be equal) by the
(potentially smaller) set $M(f)\cup M(g)$. The technical ingredient of the
proof is contained in the following lemma.

\begin{lemma}
\label{comp} For any $f,g\in \mathbb{R}^{\mathcal{V}}$ with $T[f]\geq T[g]$,
we have $L^{f}[g]\geq L^{f}[f]$.
\end{lemma}

\begin{proof}
Indeed, for any $x\in \mathcal{V}$, we have 
\begin{eqnarray*}
-L^{f}[g](x) &=&\sum_{y\,:\,f(y)\leq f(x)}L(x,y)(g(x)-g(y))=  \\
&=&\!\!\!\!\!\sum_{y\,:\,f(y)\leq f(x),\,g(y)\leq
g(x)}\!\!\!\!\!L(x,y)(g(x)-g(y))+\!\!\!\!\!\sum_{y\,:\,f(y)\leq
f(x),\,g(y)>g(x)}\!\!\!\!\!L(x,y)(g(x)-g(y)) \\
&\leq &\sum_{y\,:\,f(y)\leq f(x),\,g(y)\leq
g(x)}L(x,y)(g(x)-g(y))\leq \sum_{y\,:\,g(y)\leq
g(x)}L(x,y)(g(x)-g(y)) \\
&=&T[g](x)\leq T[f](x)=-L^{f}[f](x).
\end{eqnarray*}

\end{proof}

We are now ready to give a probabilistic proof of the following comparison
result. (Recall from~\eqref{eq:M} the definition of $M(f)$.)

\begin{proposition}
Let $f,g\in \mathbb{R}^{\mathcal{V}}$ be two functions satisfying:  \smallskip\newline 
{\rm (i).} $T[f](x)\geq T[g](x)$, for all $x\in\mathcal{V}$ ; and \smallskip\newline  
{\rm (ii).}  $f(x)\geq g(x)$, for all $x\in M(f)$ . \medskip\newline 
Then $f\geq g$.
\end{proposition}

\begin{proof}
Due to the martingale problem characterization of $X_{x}^{f}$ (see \cite{EK1986} e.g.), there exists a martingale $\{M_{g}^{f}(t)\}_{t\geq 0}$ starting from 0 such that 
\begin{equation*}
g(X_{x}^{f}(t))=g(x)+\int_{0}^{t}L^{f}[g](X_{x}^{f}(s))\,ds+M_{g}^{f}(t),\qquad\text{for all } t\geq 0.
\end{equation*}
Taking expectations we get 
\begin{equation}
\mathbb{E}[g(X_{x}^{f}(t))]=g(x)+\int_{0}^{t}
\mathbb{E}[L^{f}[g](X_{x}^{f}(s))]\,ds,\qquad\text{for all } t\geq 0. \label{eq:t}
\end{equation}

Denote by $\pi ^{f}$ the limit law of the distributions of $X_{x}^{f}(t)$
for large $t\geq 0$. Then $\pi ^{f}$ is supported on $M(f)$ and
\begin{equation*}
\lim_{t\rightarrow +\infty }\mathbb{E}[g(X_{x}^{f}(t))]=\pi ^{f}[g].
\end{equation*}

In particular the integral of the right hand side\ of \eqref{eq:t} converges
for large $t\geq 0$ and it holds:
\begin{equation}\label{eq:star}
\pi ^{f}[g]\,=\,g(x)\,+\int_{0}^{+\infty }\mathbb{E}[L^{f}[g](X_{x}^{f}(s))]\,ds.
\end{equation}

Applying the above arguments with $g$ replaced by $f$, we also obtain 
\begin{eqnarray}  \label{r2}
\pi^f[f]\,=\,f(x)+\int_0^{+\infty} \mathbb{E}[L^f[f](X^f_x(s))]\, ds.
\end{eqnarray}

The assumption (ii) yields $\pi ^{f}[f]\geq \pi ^{f}[g]$. On the other
hand, from Lemma~\ref{comp}, we have 
\begin{equation*}
\mathbb{E}[L^{f}[g](X_{x}^{f}(s))]\geq \mathbb{E}
[L^{f}[f](X_{x}^{f}(s))], \qquad\text{for all } s\geq 0.
\end{equation*}
Combining the above with~\eqref{eq:star} we deduce  
\begin{equation*}
\pi^{f}[f]\,\geq \,\pi^{f}[g]= g(x)+\int_{0}^{+\infty }\mathbb{E}[L^{f}[g](X_{x}^{f}(s))]\,\geq\,g(x)+\int_{0}^{+\infty }\mathbb{E}[L^{f}[f](X_{x}^{f}(s))]
\,ds.
\end{equation*}
Comparing the above inequality with~\eqref{r2} yields $f(x)\geq g(x)$ and the result follows.
\end{proof}

\bigskip 

By symmetry we obtain the following corollary: 

\begin{corollary} 
Let $f,g\in \mathbb{R}^{\mathcal{V}}$ be such that \smallskip\newline 
{\rm (i).} $T[f](x)= T[g](x)$, for all $x\in\mathcal{V}$ ; and \smallskip\newline  
{\rm (ii).} $f(x)= g(x)$, for all $x\in M(f)\cup M(g)$ . \smallskip\newline 
Then $f = g$.
\end{corollary}

\subsection{Classification of descent moduli on $\mathbb{R}^{\mathcal{V}}$}

\label{cm}

Denote $\mathcal{P(\cV)}^{\ast }$ the family of nonempty subsets of $\mathcal{V}$. Given a descent modulus $T$ on $\mathcal{V}$ we recall from \eqref{eq:ZT} the critical map 
\begin{equation*}
\mathcal{Z}_{T}\,:\,\mathbb{R}^{\mathcal{V}}\rightarrow \mathcal{P(V)}^{\ast}.
\end{equation*}

\begin{definition}[equivalence of descent moduli]
Let $T$, $S$ be two descent moduli on $\mathcal{V}.$ We say that the moduli $T$ and $S$
are \textit{equivalent} (and denote $T\sim S$) if $\mathcal{Z}_{T}=\mathcal{Z}_{S}$.
\end{definition}

Notice that if $T\sim S$, then  $T$ and $S$ determine the same functions via Theorem \ref{sec03:thm:Determination}.\medskip 

A family $\mathcal{D}:=\{\mathcal{D}_{x}\}_{x\in \mathcal{V}}$ is called an
\textit{active neighborhood system}, provided $x\in \mathcal{D}_{x}\subset \mathcal{V}$ for every $x\in \mathcal{V}$. An example of such system has been defined in~\eqref{eq:act} in the particular case where the set $\mathcal{V}$ is equipped with a generator $L$.\smallskip\newline
We henceforth denote by $\mathcal{E}(\mathcal{V})$ the set of active neihborhood systems on $\mathcal{V}$. Then for any such system $\mathcal{D}\in \mathcal{E}(\mathcal{V})$ we associate a descent modulus $T_{\mathcal{D}}$ defined for $f\in \mathbb{R}^{\mathcal{V}}\ $and $x\in 
\mathcal{V}$ as follows (compare with~\eqref{eq:D-global} in Proposition~\ref{prop-m-slope}):
\begin{equation}
T_{\mathcal{D}}[f](x):=\max_{y\in \mathcal{D}_{x}}\, (f(x)-f(y))_{+}
\label{eq:21}
\end{equation}

Conversely, given any (abstract) descent modulus $T$ we set:
\begin{equation}
\left\{ 
\begin{array}{rcl}
\mathcal{K}_{x}(T) & \df & \left\{ K\subset \mathcal{V}\,:\,x\in K\cap 
\mathcal{Z}_{T}(\mathds{1}_{K})\right\} 
 \\[2mm]
D_{x}(T) & \df & \bigcap_{K\in \mathcal{K}_{x}(T)}K
\end{array}
\right.   \label{eq:22}
\end{equation}
where $\mathds{1}_{K}$ denotes the characteristic function of the set $K,$
that is:
\begin{equation*}
\mathds{1}_{K}(x)=\left\{ 
\begin{tabular}{ll}
$1,$ & if  $x\in K$ \\ 
$0,$ & if $x\notin K.$
\end{tabular}
\right. 
\end{equation*}
The interest of these notions is illustrated by the following result:

\begin{theorem}[classification of moduli]
\label{caract} If a homogeneous descent modulus $T$ satisfies 
\begin{equation}
\forall \ x\in \mathcal{V}:\quad D_{x}(T) \in \mathcal{K}_{x}(T) 
\tag{$\mathcal{H}$}
\end{equation}

then there exists a family $\mathcal{D}\in \mathcal{E}(\mathcal{V})$ such
that $T$ is equivalent to $T_{\mathcal{D}}$ given in (\ref{eq:21}).
\end{theorem}

\medskip 
Before we proceed, let us introduce the following definition.
	\begin{definition} For  a critical map $\mathcal{Z}:\,\mathbb{R}^{\mathcal{V}}\rightarrow \mathcal{P(V)}^{\ast
		}$ and for each $x\in \mathcal{V}$, we define
		\begin{equation}
			\left\{ 
			\begin{array}{rcl}
				\mathcal{K}_{x}(\mathcal{Z}) & \df & \{K\subset \mathcal{V}\,:\,x\in K\cap \mathcal{Z}(\mathds{1}_{K})\} \\[2mm]
				\mathcal{D}_{x}(\mathcal{Z}) & \df & \bigcap_{K\in \mathcal{K}_{x}}K
			\end{array}
			\right.   \label{CD}
		\end{equation}
		If there is no confusion, we might simply write $\mathcal{K}_{x}$ and $\mathcal{D}_{x}$, respectively.
	\end{definition}

The proof of Theorem~\ref{caract} is based on a characterization of those
maps $\mathcal{Z}:\,\mathbb{R}^{\mathcal{V}}\rightarrow \mathcal{P(V)}^{\ast}$ for which there exists a descent modulus $T$ such that $\mathcal{Z=Z}_{T}.$
More precisely, let $\mathcal{Z}\,:\,\mathbb{R}^{\mathcal{V}}\rightarrow \mathcal{P(\cV)}^{\ast }$ be an \textit{abstract critical map} satisfying
the following conditions :\medskip \newline
(Z1) for every $f\in \mathbb{R}^{\mathcal{V}}$\ and $r\in \mathbb{R}$: $\mathcal{Z}[f+r]=\mathcal{Z}[f]$\medskip \newline
(Z2) for every $f\in \mathbb{R}^{\mathcal{V}}$ and $r>0$:\  $\mathcal{Z[}rf]=\mathcal{Z}[f]$\medskip \newline
(Z3) for every $f\in \mathbb{R}^{\mathcal{V}}$\ and $r\in \mathbb{R}$: \ \  $\mathcal{Z}[f]=\left( \mathcal{Z}[\phi _{r}(f)]\cap \lbrack f\leq
r]\right) \sqcup \left( \mathcal{Z}[\varphi _{r}(f)]\cap \lbrack f>r]\right) 
$\smallskip \newline
where 
\begin{equation*}
\forall \ r\in \mathbb{R},\,\forall \ s\in \mathbb{R},\qquad \left\{ 
\begin{array}{rcl}
\phi _{r}(s) & \df & r\wedge s \\ 
\varphi _{r}(s) & \df & r\vee s.
\end{array}
\right. 
\end{equation*}
We also assume\medskip \newline
(Z4) for every$\ K\subset \mathcal{V}$ we have: $K^{\mathrm{c}}\subset \mathcal{Z}(\mathds{1}_{K});$ and\medskip \newline
(Z5) for every $x\in \mathcal{V}:$ 
\begin{equation*}
\mathcal{K}_{x}=\{K\subset \mathcal{V}\,:\,\mathcal{D}_{x}\subset K\}.
\end{equation*}

The announced characterization of critical maps is the following:

\begin{theorem}[characterization of critical maps]
\label{critmap} An abstract critical map $\mathcal{Z}\,:\,\mathbb{R}^{\mathcal{V}}\rightarrow \mathcal{P(\cV)}^{\ast }$ is associated to some homogeneous descent modulus $T$ (that is, $\mathcal{Z=Z}_{T}$)  if and only
if conditions $(Z1)$--$(Z5)$ hold. \newline 
In this case $\mathcal{Z}=\mathcal{Z}_{T_{\mathcal{D}}}$, where $T_{\mathcal{D}}$ is 
defined by~\eqref{eq:21} for $\mathcal{D}:=\{\mathcal{D}_{x}\}_{x\in \mathcal{V}}$ constructed in~\eqref{CD}.
\end{theorem}

The last assertion of Theorem~\ref{critmap} is implicitly assuming that $\mathcal{D}\in \mathcal{E}(\mathcal{V})$. The following lemma confirms that this is indeed the case:

\begin{lemma}\label{fac} Let $\mathcal{Z}\,:\,\mathbb{R}^{\mathcal{V}}\rightarrow 
\mathcal{P(\cV)}^{\ast }$ be an abstract critical mapping that satisfies conditions $(Z1)$--$(Z4)$. Let further $\mathcal{D}:=\{\mathcal{D}_{x}\}_{x\in \mathcal{V}}$
be constructed as in~\eqref{CD}. Then $\mathcal{Z}[\mathds{1}_{\mathcal{V}}]=\mathcal{V}$ and $\mathcal{D\in E}(\mathcal{V})$.
\end{lemma}

\begin{proof}
Applying (Z4) with $K=\emptyset$, we get $\mathcal{V}\subset \mathcal{Z}(\mathds{1}_{\emptyset })=\mathcal{Z}[\boldsymbol{0}]$, where $\boldsymbol{0}$ denotes the null function on $\mathcal{V}$. We deduce from (Z1) (for $r=1$)
that $\mathcal{Z}[\mathds{1}_{\mathcal{V}}]=\mathcal{Z}[\boldsymbol{0}]=\mathcal{V}$. Recall that the family $\mathcal{D}:=(\mathcal{D}_{x})_{x\in \mathcal{V}}$
belongs to $\mathcal{E}(\mathcal{V})$ if and only if it satisfies 
\begin{equation*}
x\in \mathcal{D}_{x},\,\, \forall x\in \mathcal{V}.
\end{equation*}
Fix $x\in \mathcal{V}$. Since $\mathcal{Z}[\mathds{1}_{\mathcal{V}}]=\mathcal{V}$, we have 
$x\in \mathcal{Z}[\mathds{1}_{\mathcal{V}}]$, which in conjunction with $x\in \mathcal{V}$ yields $\mathcal{V}\in \mathcal{K}_{x}$. It follows that $\mathcal{K}_{x}\neq \emptyset $. By definition of $\mathcal{K}_{x}$, for any $K\in 
\mathcal{K}_{x}$, we have $x\in K$, so that $x\in \bigcap_{K\in \mathcal{K}_{x}}K=D_{x}$. Therefore $\mathcal{D\in E}(\mathcal{V})$.
\end{proof}

Let us postpone for a while the proof of Theorem~\ref{critmap} and show
instead that Theorem~\ref{critmap} implies Theorem~\ref{caract}. To this end, let $T$ be a
homogeneous descent modulus on $\mathcal{V}$. Then using Theorem~\ref{critmap}, it is easy to see that $T$ is equivalent to $T_{\mathcal{D}}$ for some active
neighborhood system $\mathcal{D}\in \mathcal{E}(\mathcal{V})$  provided the following proposition is proven:
\begin{proposition}
\label{verif} Under Assumption ($\mathcal{H}$), the critical map $\mathcal{Z}_{T}$ satisfies $(Z1)$--$(Z5)$.
\end{proposition}

\begin{proof}
We verify successively that conditions (Z1)--(Z5) hold. Indeed, condition (Z1) comes from the translation invariance property of $T$, see Proposition~\ref{sec03:prop:EquivalentProperties}, while (Z2)  is consequence of the homogeneity assumption for $T$. \smallskip\newline 
Verifying (Z3) requires some extra work: fix $f\in \mathbb{R}^{\mathcal{V}}$ and $r\in \mathbb{R}$. Since $(Z3)$ holds
trivially for constant functions, we may assume that $f$ takes at least two
different values. Then for any $s,s^{\prime }\in \mathbb{R}$   we have 
\begin{eqnarray*}
(\phi _{r}(s)-\phi_r (s^{\prime }))_+ &\leq &(s-s^{\prime })_+ \\
(\varphi _{r}(s)-\varphi_r (s^{\prime }))_+ &\leq &(s-s^{\prime })_+
\end{eqnarray*}
and monotonicity yields
\begin{eqnarray*}
T[\phi _{r}(f)] &\leq &T[f] \\
T[\varphi _{r}(f)] &\leq &T[f]
\end{eqnarray*}
Consequently:
\begin{equation*}
\mathcal{Z}_{T}[f]\subset \mathcal{Z}_{T}[\phi _{r}(f)]\cap \mathcal{Z}_{T}[\varphi _{r}(f)]
\end{equation*}
so that 
\begin{align*}
\mathcal{Z}_{T}[f] &=\left(\mathcal{Z}_{T}[f]\cap [f\leq r]\right)\sqcup \left(\mathcal{Z}_{T}[f]\cap [f> r]\right)\\
& \subset \left( \mathcal{Z}_{T}[\phi _{r}(f)]\cap \lbrack
f\leq r]\right) \sqcup \left( \mathcal{Z}_{T}[\varphi _{r}(f)]\cap \lbrack
f>r]\right) .
\end{align*}
To get the reserved inclusion, consider $x\in \mathcal{V}$ with $f(x)\leq r$, in particular $\phi _{r}(f(x))=f(x)$. For any $z\in \mathcal{V}$ with $\phi _{r}(f(z))<\phi _{r}(f(x))$, we have $\phi _{r}(f(z))=f(z)$, so that 
\begin{equation*}
(\phi _{r}(f(x))-\phi _{r}(f(z)))_{+}=(f(x)-f(z))_{+}
\end{equation*}
For any $z\in \mathcal{V}$ with $\phi _{r}(f(z))\geq \phi _{r}(f(x))$, we
must have $f(z)\geq f(x)$, thus 
\begin{equation*}
(\phi _{r}(f(x))-\phi _{r}(f(z)))_{+}=0\ =\ (f(x)-f(z))_{+}
\end{equation*}

From the monotonicity property, we deduce $T[\phi _{r}(f)](x)=T[f](x)$.
These considerations show that 
\begin{equation}
\mathcal{Z}_{T}[\phi _{r}(f)]\cap \lbrack f\leq r]\subset \mathcal{Z}_{T}[f]
\label{ZTsub}
\end{equation}

Finally, consider $x\in \mathcal{V}$ with $f(x)>r$, in particular $\varphi
_{r}(f(x))=f(x)$. Since $f$ is not constant, we can define
\begin{equation*}
	a:=\frac{f(x)-(r\vee \min f)}{\max f-\min f} > 0.
\end{equation*}
 
Now, on the one hand, for any $z\in \mathcal{V}$ with $\varphi _{r}(f(z))\geq
\varphi _{r}(f(x))$, we must have $f(z)\geq f(x)$, thus 
\begin{equation*}
(\varphi _{r}(f(x))-\varphi _{r}(f(z)))_{+}=0\ =\ (f(x)-f(z))_{+} =\ (af(x)-af(z))_{+}
\end{equation*}

On the other hand, for any $z\in \mathcal{V}$ with $\varphi _{r}(f(z))<\varphi _{r}(f(x))$, we
have 
\begin{align*}
\varphi _{r}(f(x))-\varphi _{r}(f(z)) \geq f(x) - (r\vee \min f) \geq a(f(x)-f(z))
\end{align*}
We deduce that 
\begin{equation*}
\forall \ z\in \mathcal{V},\qquad (\varphi _{r}(f(x))-\varphi _{r}(f(z))_{+}\geq
(af(x)-af(z))_{+}
\end{equation*}
and by monotonicity $T[\varphi _{r}(f)](x)\geq T[af](x)=aT[f](x)$, by
homogeneity. It follows that 
\begin{equation*}
\mathcal{Z}_{T}[\varphi _{r}(f)]\cap \lbrack f>r]\subset \mathcal{Z}_{T}[f].
\end{equation*}
Combining with \eqref{ZTsub}, we get the reverse inclusion 
\begin{equation*}
\mathcal{Z}_{T}[f]\supset \left( \mathcal{Z}_{T}[\phi _{r}(f)]\cap \lbrack
f\leq r]\right) \sqcup \left( \mathcal{Z}_{T}[\varphi _{r}(f)]\cap \lbrack
f>r]\right),
\end{equation*}
therefore (Z3) holds. \smallskip\newline
Condition (Z4) is a consequence of the preservation of global
minima, since the set of global minima of $\mathds{1}_{K}$ coincides with $K^{\mathrm{c}}$. \smallskip\newline
It remains to show (Z5). Set 
$$\widetilde{\mathcal{K}_x} =\{K\subset \mathcal{V}:\,\mathcal{D}_x\subset K\}.$$ 
Then for every $K\in\mathcal{K}_x$ we have $\mathcal{D}_x\subset K$, that is, $K\in\widetilde{\mathcal{K}_x}$ and $\mathcal{K}_x\subset\widetilde{\mathcal{K}_x}$.
To prove the reverse inclusion, consider $K\subset \mathcal{V}$ with $\mathcal{D}_{x}\subset K$. We need to
verify that $K\in \mathcal{K}_{x}$. Since $x\in \mathcal{D}_{x}$, we get $x\in K$. Furthermore, for every $z\in \mathcal{V}$ we have:
\begin{equation*}
\left (\mathds{1}_{\!K}(x)-\mathds{1}_{\!K}(z)\right )_{+}=1-\mathds{1}_{\!K}(z)\leq 1-\mathds{1}_{\!\mathcal{D}_{x}}(z)=(\mathds{1}_{\!\mathcal{D}_{x}}(x)-\mathds{1}_{\!\mathcal{D}_{x}}(z))_{+}
\end{equation*}
and thus by monotonicity, $T[\mathds{1}_{\!K}](x)\leq T[\mathds{1}_{\!\mathcal{D}_{x}}](x)=0$, where the last equality is obtained via~($\mathcal{H}$). It follows that $x\in \mathcal{Z}_{T}(K)$ whence $K\in \mathcal{K}_{x}$. 
\smallskip\newline
The proof is complete.
\end{proof}

For $\mathcal{D}\in \mathcal{E}(\mathcal{V})$, denote for simplicity by $\mathcal{Z}_{\mathcal{D}}$ the critical map 
$\mathcal{Z}_{T_{\mathcal{D}}}$ associated to the homogeneous descent modulus $T_{\mathcal{D}}$. The following lemma shows how to recover the active neighborhood system $\mathcal{D=}\{\mathcal{D}_{x}\}_{x}$ from 
$\mathcal{Z}_{\mathcal{D}}$:

\begin{lemma}
\label{cKD} For any $x\in \mathcal{V}$, we have 
\begin{equation*}
\mathcal{D}_{x}=\bigcap_{K\in \mathcal{K}_{x}}K
\end{equation*}
where 
\begin{equation*}
\mathcal{K}_{x}:=\{K\subset \mathcal{V}\,:\,x\in K\cap \mathcal{Z}_{\mathcal{%
D}}(\mathds{1}_{\!K})\}.
\end{equation*}
\end{lemma}

\begin{proof}
For any $f\in\RR^\mathcal{V}$, recall that
\bq
\mathcal{Z}_{\mathcal{D}}[f]&=&\{x\in \mathcal{V}\st T_{\mathcal{D}}[f](x)=0\} \,=\, \{x\in \mathcal{V}\st \max_{y\in \mathcal{D}_x}(f(x)-f(y))_+=0\}\\
&=&\{x\in \mathcal{V}\st \fo y\in \mathcal{D}_x,\, f(y)\geq f(x)\}.
\eq
\par
In particular taking $f=\ind{K}$ with $K\in\cP(\cV)^*$, we get
\[
\mathcal{Z}_{\mathcal{D}}(\mathds{1}_{\!K})=\{x\in \mathcal{V}\st \fo y\in \mathcal{D}_x,\, \ind{K}(y)\geq \ind{K}(x)\} =\{x\in K\st \mathcal{D}_x\subset K\} \cup K^c.
\]
\par
Fix $x\in \mathcal{V}$ and consider $K\in\cK_x$.
Since $x\in K$ and $x\in \mathcal{Z}_{\mathcal{D}}(\mathds{1}_{\!K})$, we deduce that $\mathcal{D}_x\subset K$.
It follows that
\bq
\mathcal{D}_x&\subset&\bigcap_{K\in\cK_x} K.\eq
\par
To get the reverse implication, it is sufficient to check that $\mathcal{D}_x\in\cK_x$.
Note that
\[
\mathcal{Z}_{\mathcal{D}}[\mathds{1}_{\mathcal{D}_x}]=\{y\in \mathcal{D}_x:\, \mathcal{D}_y\subset \mathcal{D}_x\}\cup \mathcal{D}_x^c. 
\]
Therefore, $x\in\mathcal{Z}_\mathcal{D}[\mathds{1}_{\mathcal{D}_x}]$. 
Since we also have $x\in \mathcal{D}_x$, we deduce that $\mathcal{D}_x\in\cK_x$.
\end{proof}

\medskip 
The above lemma justifies the introduction of the objects $\mathcal{K}_{x}$, $\mathcal{D}_{x}$, for $x\in \mathcal{V}$, for any mapping $\mathcal{Z}:\,\mathbb{R}^{\mathcal{V}}\rightarrow \mathcal{P(\mathcal{V})}^{\ast }$ in \eqref{CD}
by analogy to \eqref{eq:22}. Denote by $\mathcal{\hat{Z}}$ the set of mappings $\mathcal{Z}\,:\,\mathbb{R}^{\mathcal{V}}\rightarrow \mathcal{P(\mathcal{V})}^{\ast }$ satisfying $\mathcal{Z}[\mathds{1}]=\mathcal{V}$ and by $\mathcal{\hat{Z}}_{\mathcal{E}}$ the set of critical maps $\mathcal{Z}_{\mathcal{D}}$
associated to $T_{\mathcal{D}}$ with $\mathcal{D}\in \mathcal{E}(\mathcal{V})
$. Let $\mathcal{Q}$ be the mapping $\mathcal{\hat{Z}}\ni \mathcal{Z}\mapsto 
\mathcal{Z}_{\mathcal{D}}\in \mathcal{\hat{Z}}_{\mathcal{E}}$ considered
above. Lemma~\ref{cKD} shows that $\mathcal{Q}^{2}=\mathcal{Q}$, that is, $\mathcal{Q}$ is a kind of non-linear projection.\smallskip 

Let us show that in Proposition~\ref{verif} we don't need to assume ($\mathcal{H}$) if the critical map is of the form $\mathcal{Z}_\mathcal{D}$:

\begin{lemma}
\label{vD} For any $\mathcal{D}\in \mathcal{E}(\mathcal{V})$, $\mathcal{Z}_{\mathcal{D}}$ satisfies $(\mathcal{H})$ and thus $(Z5)$.
\end{lemma}

\begin{proof}

Thanks to Lemma~\ref{cKD}, the family $\mathcal{\{\bigcap }_{K\in \mathcal{K}_{x}}K\}_{x\in \mathcal{V}}$ constructed in \eqref{CD} coincides
with the active neighborhood system $\mathcal{D=}\{\mathcal{D}_{x}\}_{x\in 
\mathcal{V}}$  in the definition of $\mathcal{Z}_{\mathcal{D}}$. Thus to
check $(\mathcal{H})$, it suffices to show that $x\in \mathcal{Z}_{\mathcal{D}}[\ind{\mathcal{D}_{x}}]$, for any $x\in \mathcal{V}$, or equivalently $T_{\mathcal{D}}[\mathds{1}_{\!\mathcal{D}_{x}}](x)=0$. A direct computation
gives:
\begin{equation*}
T_{\mathcal{D}}[\mathds{1}_{\!\mathcal{D}_{x}}](x)=\max_{z\in \mathcal{D}_{x}}\mathds{1}_{\!\mathcal{D}_{x}}(x)-\mathds{1}_{\!\mathcal{D}_{x}}(z)=0
\end{equation*}

Condition (Z5) then follows from Proposition~\ref{verif}. 
\end{proof}

Here is the first step towards Theorem~\ref{critmap}:

\begin{proposition}
\label{prop:ReductionToSets} Let $\mathcal{Z}:\,\mathbb{R}^{\mathcal{V}}\rightarrow \mathcal{P(\mathcal{V})}^{\ast }$ satisfying $(Z1)$--$(Z3)$ and $\mathcal{Z}[\mathds{1}_{\mathcal{V}\!}]=\mathcal{V}$. Let $\mathcal{D}$ be
constructed as in \eqref{CD} and define $\mathcal{Z}_{\mathcal{D}}$ the
critical map associated to $T_{\mathcal{D}}$ given in \eqref{eq:21}. Assume
that 
\begin{equation*}
\forall \ K\subset \mathcal{V},\qquad \mathcal{Z}(\mathds{1}_{K})=\mathcal{Z}_{\mathcal{D}}(\mathds{1}_{K})
\end{equation*}

Then we have $\mathcal{Z}=\mathcal{Z}_{\mathcal{D}}$.
\end{proposition}

\begin{proof}
From Proposition~\ref{verif} and Lemma~\ref{fac}, $\mathcal{Z}_{\mathcal{D}}$
also satisfies (Z1)--(Z3) and $\mathcal{Z}_{\mathcal{D}}[\mathds{1}_{\mathcal{V}\!}]=\mathcal{V}$.

Let $f\in \mathbb{R}^{\mathcal{V}}$. We prove that $\mathcal{Z}[f]=
\mathcal{Z}_{\mathcal{D}}[f]$ via induction over
the number $n\in \mathbb{N}$ of values taken by $f$.

$\bullet $ We begin with the case where $n=1$, that is, $f$ is constant.
Denote by $a\in \mathbb{R}$ the value of $f$. Taking into account condition (Z1) and the fact that $\mathcal{Z}[\mathds{1}_{\mathcal{V}\!}]=\mathcal{V}$, we obtain 
\begin{equation*}
\mathcal{Z}[f]=\mathcal{Z}[f-a+1]=\mathcal{Z}[\mathds{1}_{\mathcal{V}}]=\mathcal{Z}_{\mathcal{D}}[\mathds{1}_{\mathcal{V}}]=\mathcal{V}=\mathcal{Z}_{\mathcal{D}}[f]
\end{equation*}

$\bullet $ Consider the case where $n=2$ and let $f(\mathcal{V})=\{a,b\}$
with $a<b$. Set $K:=[f=b]$. Using (Z1) and (Z2), we get 
\begin{eqnarray*}
\mathcal{Z}[f] =\mathcal{Z}\left[ \frac{f-a}{b-a}\right]  =\mathcal{Z}(\mathds{1}_{K})=\mathcal{Z}_{\mathcal{D}}(\mathds{1}_{K})=
\mathcal{Z}_{\mathcal{D}}[f].
\end{eqnarray*}

$\bullet $ Consider the case where $n>2$, assuming that $\mathcal{Z}[g]=
\mathcal{Z}_{\mathcal{D}}[g]$ for all $g\in \mathbb{R}^{\mathcal{D}}$ taking
at most $n-1$ values. Write $f_{1}<f_{2}<\cdots <f_{n}$ the values taken by $f$. Take $k=\lfloor\frac{n+1}{2}\rfloor$ (integer part), set $r=f_{k}$ and
\begin{eqnarray*}
g_{-} &\df&\phi _{r}\circ f \\
g_{+} &\df&\varphi _{r}\circ f
\end{eqnarray*}

By the choice of $r$, both $g_-$ and $g_+$ take at most $n-1$ values.

Condition (Z3) then yields
\begin{eqnarray*}
\mathcal{Z}[f] &=&\left( \mathcal{Z}[g_{-}]\cap \lbrack f\leq r]\right)
\sqcup \left( \mathcal{Z}[g_{+}]\cap \lbrack f>r]\right)  \\
&=&\left( \mathcal{Z}_{\mathcal{D}}[g_{-}]\cap \lbrack f\leq r]\right)
\sqcup \left( \mathcal{Z}_{\mathcal{D}}[g_{+}]\cap \lbrack f>r]\right) =
\mathcal{Z}_{\mathcal{D}}[f].
\end{eqnarray*}
as desired.
\end{proof}

Having established Proposition~\ref{prop:ReductionToSets}, the following
result finishes the proof of Theorem~\ref{critmap}:

\begin{proposition}
Let $\mathcal{Z}:\,\mathbb{R}^{\mathcal{V}}\rightarrow \mathcal{P(\cV)}^{\ast }$  be a mapping satisfying $(Z1)$--$(Z5)$ and let $\mathcal{D}=(\mathcal{D}_{x})_{x\in \cV}\in \mathcal{E}(\mathcal{V})$ be as in~\eqref{CD}.
Then we have 
\begin{equation*}
\forall \ K\subset \mathcal{V}:\qquad \mathcal{Z}[\mathds{1}_{K}]=\mathcal{Z}_{\mathcal{D}}[\mathds{1}_{K}]
\end{equation*}
\end{proposition}

\begin{proof}
From (Z4), we have 
\begin{equation*}
K^{\mathrm{c}}\cap \mathcal{Z}(\mathds{1}_{K})=K^{\mathrm{c}}=K^{\mathrm{c}}\cap \mathcal{Z}_{\mathcal{D}}(\mathds{1}_{K}).
\end{equation*}
Now, let $K^{\prime }=\{x\in K\ :\ \mathcal{D}_{x}\subset K\}$. For every $x\in K^{\prime }$, due to (Z5), we have $K\in \mathcal{K}_{x}$, so $x\in \mathcal{Z}[\mathds{1}_{K}]$. Since this is true for any $x\in
K^{\prime }$, we get $K^{\prime }\subset \mathcal{Z}[\mathds{1}_{K}]$.

According to Lemma~\ref{vD}, $\mathcal{Z}_{\mathcal{D}}$ also satisfies 
(Z5), and it follows as above that $K^{\prime }\subset \mathcal{Z}_{\mathcal{D}}[\ind{K}]$. We deduce 
\begin{equation*}
K^{\prime }\cap \mathcal{Z}(\mathds{1}_{K})=K^{\prime }\ =\ K^{\prime }\cap 
\mathcal{Z}_{\mathcal{D}}(\mathds{1}_{K}).
\end{equation*}

Finally, let $K\diagdown K^{\prime }=\{x\in K\ :\ \mathcal{D}_{x}\setminus
K\neq \emptyset \}$. For every $x\in K\diagdown K^{\prime }$, since $\mathcal{D}_{x}\not\subset K$, we have $K\notin \mathcal{K}_{x}$, due to the
definition of $\mathcal{K}_{x}$ in \eqref{CD}. Since $x\in K$, the only
possibility is that $x\notin \mathcal{Z}(\mathds{1}_{K})$.

The same reasoning applies to $\mathcal{Z}_{D}$ (recalling Lemma~\ref{cKD})
and we get 
\begin{equation*}
K\diagdown K^{\prime }\cap \mathcal{Z}(\mathds{1}_{K})=\emptyset \ =\
K\diagdown K^{\prime }\cap \mathcal{Z}_{\mathcal{D}}(\mathds{1}_{K})
\end{equation*}

Since $\mathcal{V}=K^{\mathrm{c}}\sqcup K^{\prime }\sqcup( K\diagdown
K^{\prime })$, we conclude that 
\begin{equation*}
\mathcal{Z}(\mathds{1}_{K})=K^{\mathrm{c}}\sqcup K^{\prime }=\mathcal{Z}_{\mathcal{D}}(\mathds{1}_{K}),
\end{equation*}
finishing the proof.
\end{proof}

\begin{remark}
Note that (Z5) was only used to prove that $K^{\prime }\subset \mathcal{Z}(\mathds{1}_{K})$. Thus, when (Z5) is not verified, the constructed modulus
of descent $T_{\mathcal{D}}$ might enlarge the critical map $\mathcal{Z}_{\mathcal{D}}$ with respect to $\mathcal{Z}$, as illustrated by Example~\ref{exafin} below.
\end{remark}

The following two examples show that  homogeneity and~(Z4) are necessary assumptions.

\begin{example}[A non-homogeneous descent modulus failing (Z2)]
Let $\varepsilon >0$ and consider the operator $T_{\varepsilon }:\mathbb{R}^{\mathcal{V}}\rightarrow \mathbb{R}_{+}^{\mathcal{V}}$ given by 
\[
T_{\varepsilon }[f](x) \df \left\{ 
\begin{array}{cl}
f(x)-\min f & \text{ if }f(x)>\min f+\varepsilon  \\[2mm]
0 & \text{ if }f(x)\leq \min f+\varepsilon .
\end{array}
\right.
\,\,=\,\,\phi _{\epsilon }(f(x)-\min f)
\]
where the mapping $\phi _{\epsilon }$ is defined for $r\geq 0$ by 
\begin{equation*}
\phi _{\epsilon }(r)\df\left\{ 
\begin{array}{ll}
0,\; & \text{if }r\in \lbrack 0,\epsilon ] \\ 
r,\; & \text{if }r>\epsilon. 
\end{array}
\right. 
\end{equation*}
We claim that $T_{\varepsilon }$ is a descent modulus. 
\begin{itemize}
\item Let $f\in \mathbb{R}^{\mathcal{V}}$. For every $x\in \argmin f$, we
have that $T_{\varepsilon }[f](x)=0$, and so $T_{\varepsilon }$ preserves
global minima.

\item Let $f,g\in \mathbb{R}^{\mathcal{V}}$ and $x\in \mathcal{V}$ such that 
\begin{equation*}
(f(x)-f(z))_{+}\geq (g(x)-g(z))_{+},\quad \forall z\in \mathcal{V}.
\end{equation*}
On the one hand, if $f(x)\leq \min f + \varepsilon$, we have that
\[
\varepsilon\geq (f(x)-f(z))_{+}\geq (g(x)-g(z))_{+},\quad \forall z\in \mathcal{V},
\]
and so, $g(x)\leq \min g + \varepsilon$ as well. Then $T_{\varepsilon }[f](x) = T_{\varepsilon }[g](x)$. On the other hand, if $f(x)\geq \min f + \varepsilon$, by  taking $z^{\ast }\in \argmin g$, we have that 
\begin{align*}
T_{\varepsilon }[f](x)& =f(x)-\min f\geq (f(x)-f(z^{\ast }))_{+} \\
& \geq (g(x)-g(z^{\ast }))_{+}=g(x)-\min g \geq T_{\varepsilon }[g](x).
\end{align*}
Thus, $T_{\varepsilon }$ is monotone.

\item Let $f\in \mathbb{R}^{\mathcal{V}}$ and $r>1$. Then, for every $x\in 
\mathcal{V}$, 
\begin{equation*}
T_{\varepsilon }[f](x)>0\implies \epsilon \leq T_{\varepsilon
}[f](x)=f(x)-\min f<r(f(x)-\min f)=rf(x)-\min rf=T_{\varepsilon }[rf](x).
\end{equation*}
Thus, $T_{\varepsilon }$ is scalar-monotone.
\end{itemize}

Then, by definition, $T_{\varepsilon }$ is a descent modulus. However,
choose $f\in \mathbb{R}^{\mathcal{V}}$ such that $\alpha =\max f-\min
f>\varepsilon $, and choose $r=\frac{\varepsilon }{\alpha }$. Then, we have
that $\mathcal{Z}_{T}(f)\neq \mathcal{V}$ but 
\begin{equation*}
\forall x\in \mathcal{V},\quad rf(x)-\min rf\leq r(\max f-\min
f)=\varepsilon \implies \mathcal{Z}_{T}(rf)=\mathcal{V}.
\end{equation*}
\hfill$\diamond$
\end{example}

\begin{example}[A family of homogeneous moduli of descent failing~(Z4)]

\label{exafin} Let $\mathcal{D}^{\prime }=(\mathcal{D}_{x}^{\prime })_{x\in
\cV}\in \mathcal{E}(\mathcal{V)}$ such that there exists $\bar{x}\in \cV$ where $|\mathcal{D}_{\bar{x}}^{\prime }|\geq 3$, and let $T:\mathbb{R}^{\mathcal{V}
}\rightarrow \mathbb{R}_{+}^{\mathcal{V}}$ be the operator given by 
\begin{equation*}
T[f](x)\df\left\{ 
\begin{array}{cl}
\left( f(x)-\displaystyle\max_{y\in \mathcal{D}_{x}^{\prime }\setminus
\{x\}}f(y)\right) _{+} & \text{ if }\mathcal{D}_{x}^{\prime }\neq \{x\} \\ 
0 & \text{ if }\mathcal{D}_{x}^{\prime }=\{x\}.
\end{array}
\right. 
\end{equation*}
Clearly $T$ is homogeneous and preserves global minima. Let us prove that $T$
is monotone: let $f,g\in \mathbb{R}^{\mathcal{V}}$ and $x\in \mathcal{V}$
such that 
\begin{equation*}
(f(x)-f(z))_{+}\geq (g(x)-g(z))_{+},\quad \forall z\in \mathcal{V}.
\end{equation*}
If $\mathcal{D}_{x}^{\prime }=\{x\}$, then $T[f](x)=0\geq 0=T[g](x)$. If $\mathcal{D}_{x}^{\prime }\neq \{x\}$, then there exist $y^{\ast }\in 
\mathcal{D}_{x}^{\prime }\setminus \{x\}$ such that $f(y^{\ast })=\max_{y\in 
\mathcal{D}_{x}^{\prime }\setminus \{x\}}f(y)$. Then, 
\begin{equation*}
T[f](x)=(f(x)-f(y^{\ast }))_{+}\geq (g(x)-g(y^{\ast }))_{+}\geq \left(
g(x)-\max_{y\in \mathcal{D}_{x}^{\prime }\setminus \{x\}}g(y)\right)
_{+}=T[g](x).
\end{equation*}
Thus, $T$ is monotone and therefore it is a descent modulus. Now, let $(\mathcal{K}_{x})_{x\in \mathcal{V}}$ and $\mathcal{D}=(\mathcal{D}_{x})_{x\in \mathcal{V}}$ constructed as in \eqref{CD} for $\mathcal{Z}_{T}$. Then:

\begin{itemize}
\item If $\mathcal{D}_{x}^{\prime }=\{x\}$, then $\mathcal{K}_{x}=\{K\in \cP(\mathcal{V})^*\ :\ x\in K\}$ and so $\mathcal{D}_{x}=\mathcal{D}_{x}^{\prime }
$.

\item If $\mathcal{D}_{x}^{\prime }=\{x,y\}$ for some $y\neq x$, then $\mathcal{K}_{x}=\{K\in \mathcal{V}\,:\,\{x,y\}\subset K\}$. Indeed,
consider $K\in \mathcal{K}_{x}$, we have $x\in K$ and $x\in \mathcal{Z}_{T}(K)$. We compute 
\begin{equation*}
T(\mathds{1}_{\!K})(x)=(\mathds{1}_{\!K}(x)-\mathds{1}_{\!K}(y))_{+}=1-\mathds{1}_{\!K}(y)
\end{equation*}
so for this expression to vanish, we must have $y\in K$. Conversely, if $\{x,y\}\subset K$, then $x\in K$ and 
\begin{equation*}
T(\mathds{1}_{\!K})(x)=(\mathds{1}_{\!K}(x)-\mathds{1}_{\!K}(y))_{+}=0
\end{equation*}
so $K\in \mathcal{K}_{x}$. We deduce $\mathcal{D}_{x}=\{x,y\}$ and so $\mathcal{D}_{x}=\mathcal{D}_{x}^{\prime }$.

\item If $|\mathcal{D}_{x}^{\prime }|\geq 3$, we have that there is $y,z\in 
\mathcal{D}_{x}^{\prime }\setminus \{x\}$ with $y\neq z$ such that $\{x,y\},\{x,z\}\in \mathcal{K}_{x}$. Thus, $\mathcal{D}_{x}\subset
\{x,y\}\cap \{x,z\}=\{x\}\neq \mathcal{D}_{x}^{\prime }$, it follows that $\mathcal{D}_{x}=\{x\}$. Furthermore, $\{x\}\notin \mathcal{K}_{x}$ since $T[\mathds{1}_{\!\{x\}}](x)=1$.
\end{itemize}

Thus, since $|\mathcal{D}_{\bar{x}}^{\prime }|\geq 3$, $T$ fails (Z4).\hfill$\diamond$
\end{example}


\noindent\rule{4cm}{1.2pt}

\medskip

\textbf{Acknowledgement} A major part of this work has been realized during a research stay of the first and the third author to Toulouse School of Economics, France (May 2022). These two authors wish to thank J. Bolte and the host institute for hospitality. The first author also acknowledges support from the Austrian Science Fund (FWF, P-36344-N).


\begin{center}
\noindent\rule{4cm}{1.4pt}
\end{center}



\begin{thebibliography}{99}
	
	\bibitem{AGS2008}
	L.~Ambrosio, N.~Gigli, and G.~Savar\'{e}.
	\newblock {\em Gradient flows in metric spaces and in the space of probability
		measures}.
	\newblock Lectures in Mathematics. Birkh\"{a}user, second edition, 2008.
	
	\bibitem{Baillon2018}
	J.-B. Baillon.
	\newblock Personal communication, October 2018.
	
	\bibitem{BGL2014}
	D.~Bakry, I.~Gentil, and M.~Ledoux.
	\newblock {\em Analysis and geometry of {M}arkov diffusion operators}.
	\newblock Grundlehren der mathematischen Wissenschaften 346. Springer, Cham,
	2014.
	
	\bibitem{BD2015}
	J.~Benoist and A.~Daniilidis.
	\newblock Subdifferential representation of convex functions: refinements and
	applications.
	\newblock {\em J. Convex Anal.}, 12(2):255--265, 2005.
	
	\bibitem{Blumenson1960Coordinates}
	L.~E. Blumenson.
	\newblock Classroom {N}otes: {A} {D}erivation of {$n$}-{D}imensional
	{S}pherical {C}oordinates.
	\newblock {\em Amer. Math. Monthly}, 67(1):63--66, 1960.
	
	\bibitem{BCD2018}
	T.~Z. Boulmezaoud, P.~Cieutat, and A.~Daniilidis.
	\newblock Gradient flows, second-order gradient systems and convexity.
	\newblock {\em SIAM J. Optim.}, 28(3):2049--2066, 2018.
	
	\bibitem{BrezisOperateurs1973}
	H.~Br\'{e}zis.
	\newblock {\em Op\'{e}rateurs maximaux monotones et semi-groupes de
		contractions dans les espaces de {H}ilbert}.
	\newblock North-Holland Publishing Co., 1973.
	
	\bibitem{CDM1993Deformation}
	J.~N. Corvellec, M.~Degiovanni, and M.~Marzocchi.
	\newblock Deformation properties for continuous functionals and critical point
	theory.
	\newblock {\em Topol. Methods Nonlinear Anal.}, 1(1):151--171, 1993.
	
	\bibitem{DS2022}
	A.~Daniilidis and D.~Salas.
	\newblock A determination theorem in terms of the metric slope.
	\newblock {\em Proc. Amer. Math. Soc.}, 150(10):4325--4333, 2022.
	
	\bibitem{GMT1980}
	E.~De~Giorgi, A.~Marino, and M.~Tosques.
	\newblock Problems of evolution in metric spaces and maximal decreasing curve.
	\newblock {\em Atti Accad. Naz. Lincei Rend. Cl. Sci. Fis. Mat. Nat. (8)},
	68(3):180--187, 1980.
	
	\bibitem{DM1994Critical}
	M.~Degiovanni and M.~Marzocchi.
	\newblock A critical point theory for nonsmooth functionals.
	\newblock {\em Ann. Mat. Pura Appl. (4)}, 167:73--100, 1994.
	
	\bibitem{DIL2015}
	D.~Drusvyatskiy, A.~D. Ioffe, and A.~S. Lewis.
	\newblock Curves of descent.
	\newblock {\em SIAM J. Control Optim.}, 53(1):114--138, 2015.
	
	\bibitem{EK1986}
	S.~N. Ethier and T.~G. Kurtz.
	\newblock {\em Markov processes}.
	\newblock Wiley Series in Probability and Mathematical Statistics. New York,
	1986.
	
	\bibitem{EG2015Measure}
	L.~C. Evans and R.~F. Gariepy.
	\newblock {\em Measure theory and fine properties of functions}.
	\newblock Textbooks in Mathematics. CRC Press, Boca Raton, FL, revised edition,
	2015.
	
	\bibitem{GZ1992}
	L.~Gajek and D.~Zagrodny.
	\newblock Countably orderable sets and their applications in optimization.
	\newblock {\em Optimization}, 26(3-4):287--301, 1992.
	
	\bibitem{Ioffe2017Variational}
	A.~D. Ioffe.
	\newblock {\em Variational analysis of regular mappings}.
	\newblock Springer Monographs in Mathematics. Springer, Cham, 2017.
	\newblock Theory and applications.
	
	\bibitem{Moreau1965}
	J.-J. Moreau.
	\newblock Proximit\'{e} et dualit\'{e} dans un espace hilbertien.
	\newblock {\em Bull. Soc. Math. France}, 93:273--299, 1965.
	
	\bibitem{PSV2021}
	P.~P\'{e}rez-Aros, D.~Salas, and E.~Vilches.
	\newblock Determination of convex functions via subgradients of minimal norm.
	\newblock {\em Math. Program.}, 190(1-2, Ser. A):561--583, 2021.
	
	\bibitem{Rockafellar1966}
	R.~T. Rockafellar.
	\newblock Characterization of the subdifferentials of convex functions.
	\newblock {\em Pacific J. Math.}, 17:497--510, 1966.
	
	\bibitem{Rock70}
	R.~T. Rockafellar.
	\newblock On the maximal monotonicity of subdifferential mappings.
	\newblock {\em Pacific J. Math.}, 33:209--216, 1970.
	
	\bibitem{RockafellarWets1998}
	R.~T. Rockafellar and Roger J.-B. Wets.
	\newblock {\em Variational analysis}.
	\newblock Grundlehren der mathematischen Wissenschaften 317. Springer, 1998.
	
	\bibitem{SV1989}
	D.~J. Smith and M.~K. Vamanamurthy.
	\newblock How small is a unit ball?
	\newblock {\em Math. Mag.}, 62(2):101--107, 1989.
	
	\bibitem{Sturm1998}
	K.~T. Sturm.
	\newblock Diffusion processes and heat kernels on metric spaces.
	\newblock {\em Ann. Probab.}, 26(1):1--55, 1998.
	
	\bibitem{TZ2022}
	L.~Thibault and D.~Zagrodny.
	\newblock Determining functions by slopes.
	\newblock {\em Comm. Contemp. Math.}, 2022.
	\newblock (In press).
	
\end{thebibliography}

\begin{center}
\noindent\rule{4cm}{1.4pt}
\end{center}

\newpage

\noindent Aris DANIILIDIS

\medskip
\noindent Institute of Statistics and Mathematical Methods in Economics,
VADOR E105-04 \newline TU Wien, Wiedner Hauptstra{\ss }e 8, A-1040 Wien\medskip
\newline(on leave) DIM-CMM, CNRS IRL 2807 \newline Beauchef 851, FCFM,
Universidad de Chile \medskip\newline\noindent E-mail:
\texttt{aris.daniilidis@tuwien.ac.at}\newline\noindent
\texttt{https://www.arisdaniilidis.at/}

\medskip

\noindent Research supported by the grants: \smallskip\newline Austrian Science Fund (FWF P-36344N) (Austria)\\ 
CMM  FB210005 BASAL funds for centers of excellence (ANID-Chile)\newline

\medskip

\noindent Laurent MICLO

\medskip

\noindent Toulouse School of Economics, University Toulouse 1 (Capitole) \newline
1, Esplanade de l'Universit\'{e}, F-31080 Toulouse Cedex 06, France \medskip\newline\noindent E-mail:
\texttt{laurent.miclo@math.cnrs.fr}\newline\noindent
\texttt{https://perso.math.univ-toulouse.fr/miclo/} 

\medskip

\noindent Research funded by the Agence Nationale de la Recherche under grant ANR-17-EURE-0010 (Investissements d'Avenir program).\newline\vspace{0.4cm}

\noindent David SALAS

\medskip

\noindent Instituto de Ciencias de la Ingenieria, Universidad de
O'Higgins\newline Av. Libertador Bernardo O'Higgins 611, Rancagua, Chile
\smallskip

\noindent E-mail: \texttt{david.salas@uoh.cl} \newline\noindent
\texttt{http://davidsalasvidela.cl} \medskip

\noindent Research supported by the grant: \smallskip\newline CMM ACE210010 and FB210005 BASAL funds for centers of excellence (ANID-Chile)\\ FONDECYT 11220586 (Chile)

\end{document}